\documentclass[10pt,a4paper]{amsart}
\usepackage{fullpage} 
\usepackage{amsmath}
\usepackage{amsthm}
\usepackage{amssymb}
\usepackage{mathtools}
\usepackage[only,mapsfrom]{stmaryrd}
\usepackage{graphicx}
\usepackage{color}
\usepackage{enumitem}
\usepackage[colorlinks,citecolor=blue,linkcolor=blue,urlcolor=blue,filecolor=blue,breaklinks]{hyperref}
\usepackage{tensor}

\numberwithin{equation}{section}
\numberwithin{figure}{section}

\newtheorem{thm}{Theorem}[section]
\newtheorem{athm}{Theorem}
\newtheorem{aprop}[athm]{Proposition}

\newtheorem{lem}[thm]{Lemma}
\newtheorem{prop}[thm]{Proposition}
\newtheorem{cor}[thm]{Corollary}
\newtheorem{ques}[thm]{Question}
\newtheorem*{thm*}{Theorem}

\theoremstyle{definition}
\newtheorem{rem}[thm]{Remark}
\newtheorem{defn}[thm]{Definition}
\newtheorem{ex}[thm]{Example}

\newcommand{\cC}{\mathcal{C}}

\newcommand{\cG}{\mathcal{G}}

\newcommand{\cU}{\mathcal{U}}
\newcommand{\cV}{\mathcal{V}}
\newcommand{\cW}{\mathcal{W}}

\newcommand{\bN}{\mathbb{N}}

\newcommand{\bZ}{\mathbb{Z}}

\newcommand{\V}{\mathrm{V}}
\newcommand{\Homeo}{\mathrm{Homeo}}

\renewcommand{\geq}{\geqslant}
\renewcommand{\leq}{\leqslant}

\title{Embedding groups into acyclic groups}

\author{Martin Palmer}
\address{Institutul de Matematică Simion Stoilow al Academiei Române, 21 Calea Griviței, 010702 Bucharest, Romania}
\email{mpanghel@imar.ro}

\author{Xiaolei Wu}
\address{Shanghai Center for Mathematical Sciences, Jiangwan Campus, Fudan University, No.2005 Songhu Road, Shanghai, 200438, P.R. China}
\email{xiaoleiwu@fudan.edu.cn}

\date{19 October 2025}

\begin{document}

\begin{abstract}
We show that labelled Thompson groups and twisted Brin--Thompson groups are all acyclic. This allows us to prove several new embedding results for groups. First, every group of type $F_n$ embeds quasi-isometrically as a subgroup of an acyclic group of type $F_n$ that has no proper finite-index subgroups. This improves results of Baumslag--Dyer--Heller ($n=1$) and Baumslag--Dyer--Miller ($n=2$) from the early 80s, as well as a more recent result of Bridson ($n=2$). Second, we show that every finitely generated group embeds quasi-isometrically as a subgroup of a $2$-generated, simple, acyclic group. Our results also allow us to produce, for each $n\geq 2$, the first known example of an acyclic group that is of type $F_n$ but not $F_{n+1}$. These examples can moreover be taken to be simple. Furthermore, our examples provide a rich source of universally boundedly acyclic groups.
\end{abstract}
\maketitle

\section*{Introduction}

Recall that a (discrete) group $G$ is called \emph{acyclic} if it has the homology of a point. There is long-standing interest in acyclic groups. Notably, Mather used a clever trick in \cite{Mather1971} to prove that the compactly-supported homeomorphism groups of finite-dimensional Euclidean spaces are acyclic. The idea of the proof leads to the definition of \emph{mitotic groups} \cite{BauDyHe80}, \emph{pseudo-mitotic groups} \cite{Varadarajan85} and \emph{binate groups} \cite{Berrick1989}, all of which are acyclic. De la Harpe and McDuff proved in \cite{dlHM83} that many automorphism groups of large structures are acyclic, including the permutation group of any infinite set.

A recent breakthrough of Szymik and Wahl \cite{SW19} shows that the Thompson group $\V$ is acyclic. Their proof involves heavy algebro-topological machinery including homological stability, stable homotopy theory and algebraic K-theory. Their results have subsequently been extended in \cite{Li2024, KLM+24, PalmerWu2024}, showing, for example, that the Brin--Thompson groups and many big mapping class groups are acyclic.

In a different direction, Kan and Thurston proved the following beautiful theorem in \cite{KanThurston76}: every connected CW-complex has the homology of a group. A key step in their proof of this theorem is (a simplicial version of) the following embedding result \cite[Proposition 3.3]{KanThurston76}: every group $G$ can be embedded into an acyclic group $C'G$ of the same cardinality (except when $G$ is finite, in which case $C'G$ is at most countable). This was later improved by Baumslag--Dyer--Heller, who proved in \cite{BauDyHe80} that every finitely generated group (and hence every countable group by \cite[Theorem IV]{HigmanNeumannNeumann49}) embeds into a finitely generated acyclic group. This was further extended by Baumslag--Dyer--Miller \cite{BauDyMi83}, who showed that every finitely presented group embeds into a finitely presented acyclic group.

Further results regarding embedding groups into acyclic groups were proven in \cite{BerrickMiller92, BerrickChatterjiMislin04, Berrick11}. We note that, in these embeddings, the host groups may be assumed, at most, to be finitely presented. It is worth mentioning that, in these papers, the authors also found applications to several old conjectures of embedding groups into acyclic groups, including resolving the Bass Conjecture for amenable groups \cite[Theorem 1.2]{BerrickChatterjiMislin04}. See also Chatterji--Mislin \cite{ChatterjiMislin03} for an another application, in their proof of Atiyah's $L^2$-index theorem.

These embedding results and applications give rise to the following question, which is the natural generalisation of this line of investigation.

\begin{ques}
\label{ques:acyclic-emb-F_n}
Does every group of type $F_n$ embed into an acyclic group of type $F_n$?
\end{ques}

Recall that a group $G$ is said to be of \emph{type $F_n$} (resp.\ of \emph{type $F$}) if there is an aspherical CW-complex with fundamental group $G$ and finite $n$-skeleton (resp.\ finitely many cells). It is of \emph{type $F_\infty$} if it is of type $F_n$ for each $n\geq 1$.

\subsection*{Labelled Thompson groups.}

The construction of Thompson's group $\V$ generalises to associate, to any discrete group $G$, a corresponding \emph{labelled Thompson group} $\V(G)$ together with a natural embedding $\iota_0 \colon G \hookrightarrow \V(G)$ (see \S\ref{subsection:label-thom} for the precise definition). This construction is natural in the sense that it forms a functor
\begin{equation*}
\V \colon \mathbf{groups} \longrightarrow \mathbf{groups}.
\end{equation*}
This embedding has already been well studied in \cite{Thompson80, WuWuZhaoZhou2025}, whose results we summarise here.

\begin{thm*}[\cite{Thompson80, WuWuZhaoZhou2025}]
\label{thm:labelled-V}
The functor $\V$ and the embedding $\iota_0$ have the following properties, for any discrete group $G$:
    \begin{enumerate}
        \item $\iota_0 \colon G \hookrightarrow \V(G)$ is injective, and is also a quasi-isometric embedding when $G$ is finitely generated;
        \item $\iota_0 \colon G \hookrightarrow \V(G)$ is Frattini;
        \item $\V(G)$ has solvable word problem if and only if $G$ has solvable word problem;
        \item $\V(G)$ is of type $F_n$ if and only if $G$ is of type $F_n$;
        \item $\V(G)$ has no proper finite-index subgroups;
        \item $\V(G)$ is $5$-uniformly perfect;
        \item $\V(G)$ is boundedly acyclic.
    \end{enumerate}
\end{thm*}

The first goal of our paper is to add the following key property of $\V(G)$ to this list.

\begin{athm}[Theorem \ref{thm:laV-acyc}]
\label{thm:VG}
For any discrete group $G$, the labelled Thompson group $\V(G)$ is acyclic.
\end{athm}

We also observe two additional properties of the functor $\V$. We recall that a group is called \emph{strongly torsion generated} if, for each $n\geq 2$, it is normally generated by a single element of order $n$.

\begin{aprop}[Propositions \ref{prop:centre-VG} and \ref{prop:torsion-VG}]
\label{aprop:VG-centre}
For any discrete group $G$, we have:
\begin{itemize}
    \item the centre of $\V(G)$ is isomorphic to the centre of $G$;
    \item the group $\V(G)$ is strongly torsion generated.
\end{itemize}
\end{aprop}

These new properties allow us to deduce many new results related to acyclic groups, in particular answering Question \ref{ques:acyclic-emb-F_n} in the positive.

\begin{cor}
Any group of type $F_n$ embeds quasi-isometrically as a subgroup of an acyclic group of type $F_n$ that has no proper finite-index subgroups.
\end{cor}

\begin{rem}
Without the acyclicity property, this was first established by Bridson \cite{Bridson98}. Our embedding has the advantage that it is functorial and explicit. Recall also that the first acyclic group with no proper finite-index subgroups was constructed by Higman \cite{Higman51}. Further examples of the same flavour were produced by Bridson and Grunewald in \cite[\S 4]{BridsonGrunewald04}, which were used as an initial input in their solution to Grothendieck's problem on profinite rigidity. In contrast to our examples, these acyclic groups are of finite cohomological dimension (they have geometric dimension $2$). For more examples of acyclic groups of finite cohomological dimension, see \cite{BerrickHillman03}. 
\end{rem}

By Higman's celebrated embedding theorem \cite{Higman1961}, every recursively presented group embeds into a universal finitely presented group. Applying Theorem \ref{thm:VG} and Proposition \ref{aprop:VG-centre} to this group, we deduce the following corollary, which improves several results in the literature, including \cite[Theorem E]{BauDyMi83}, \cite[Theorem 6]{BerrickMiller92} and \cite[Theorem A]{Bridson20}.

\begin{cor}
There is a finitely presented, strongly torsion generated, acyclic group with no proper finite-index subgroups into which every recursively presented group embeds.
\end{cor}

Given any abelian group $A$, Baumslag--Dyer--Heller constructed an acyclic group with centre isomorphic to $A$ \cite[Theorem 7.1]{BauDyHe80}. In the special case when $A = \bZ^n$ for any $n\geq 1$, they also constructed \emph{finitely generated} groups $G$ with centre $\bZ^n$ and $H_1(G)=H_2(G) = \{0\}$ \cite[Theorem 7.4]{BauDyHe80}. Applying Theorem \ref{thm:VG} and Proposition \ref{aprop:VG-centre} with $G=A$, we improve their results to the following.

\begin{cor}
Any abelian group $A$ is isomorphic to the centre of an acyclic, strongly torsion generated group $\Gamma$. When $A$ is finitely generated, $\Gamma$ can be taken to be of type $F_\infty$.
\end{cor}

\begin{rem}
This also improves a result of Berrick \cite[Theorem A]{Berrick91}, who showed that every abelian group $A$ is isomorphic to the centre of a perfect, strongly torsion generated group $\Gamma$.
\end{rem}

More recently, Ould Houcine proved in \cite{Ould07} that every recursively presented abelian group $A$ may be realised as the centre of a finitely presented group $G$. By embedding $G$ further into $\V(G)$, we have the following result.

\begin{cor}
Every recursively presented abelian group $A$ may be realised as the centre of a finitely presented acyclic group.
\end{cor}

In a different direction, Collins and Miller \cite{CollinsMiller99} constructed a group with compact $2$-dimensional classifying space, hence of type $F$ (and thus also type $F_\infty$), that has unsolvable word problem. Taking the labelled group $G$ in Theorem \ref{thm:VG} to be their group, we have the following interesting application. 

\begin{cor}
There exists an acyclic group of type $F_\infty$ that has unsolvable word problem and no proper finite index subgroups.
\end{cor}

One way to show that a group $G$ is not of type $F_{n+1}$ is to show that $H_{n+1}(G;\bZ)$ is not a finitely generated abelian group. This is, for example, how Stallings \cite{Stallings63} found the first group of type $F_2$ but not $F_3$. This idea was later generalised by Bieri \cite[Proposition 4.1]{Bieri76} to find, for each $n\geq 1$, groups of type $F_n$ but not of type $F_{n+1}$. However, this approach clearly cannot be used if one is dealing with acyclic groups, and it is an interesting question to find acyclic groups with exotic finiteness properties. The \emph{Bestvina--Brady groups} \cite{BestvinaBrady97} provide a rich class of groups with exotic finiteness properties (indeed, in retrospect the groups constructed by Stallings and by Bieri are examples of Bestvina--Brady groups), but they are all locally indicable and hence cannot be acyclic. However, by taking $G$ to be any group of type $F_n$ but not $F_{n+1}$ in Theorem \ref{thm:VG}, we have the following.

\begin{cor}
\label{cor:acyc-Fn-nFn+1}
For any $n\geq 1$, there is an acyclic group of type $F_n$ but not $F_{n+1}$.
\end{cor}

In the case $n=1$, such a group may also be constructed by embedding any finitely generated but not recursively presented group into a finitely generated acyclic group using \cite{BauDyHe80}. However, to the best of our knowledge, for each $n \geq 2$ no groups with such properties were previously known. In the case $n=2$, it is likely that the universal finitely presented acyclic group arising from \cite[Theorem E]{BauDyMi83} is not of type $F_3$, but we are not aware of a proof.

\subsection*{Twisted Brin--Thompson groups.}

Our next goal is to prove that the twisted Brin--Thompson groups are also acyclic. These groups were first constructed by Belk and Zaremsky \cite{BelkZaremsky22} as a generalisation of Brin's higher-dimensional Thompson groups \cite{Brin04}. Given any group $G$ acting faithfully on a countable set $S$, they defined an associated \emph{twisted Brin--Thompson group}, denoted by $S\V_G$. The key property of the groups $S\V_G$ is that they are always simple \cite[Theorem 3.4]{BelkZaremsky22}. In fact, they provide a flexible way to produce simple groups with good finiteness properties.

\begin{athm}[Theorem \ref{thm:tbt-acyc}]
\label{thm:acyclic-t-br-thom}
The twisted Brin--Thompson groups are acyclic. 
\end{athm}

This theorem also holds even if we do not assume that the underlying set $S$ is countable, as explained in Remark \ref{rmk:uncountable-S}. Combining Theorem \ref{thm:acyclic-t-br-thom} with \cite[Theorems A, B, 3.4]{BelkZaremsky22} and taking $S=G$ with the regular $G$-action, we have the following.

\begin{cor}
Every finitely generated group embeds quasi-isometrically into a $2$-generated, simple, acyclic group. 
\end{cor}

Taking $G$ to be the $n$th Houghton group acting on the disjoint union of $n$ copies of the natural numbers $S = \{1,\ldots,n\} \times \bN$, the corresponding twisted Brin--Thompson group $S\V_G$ is of type $F_{n-1}$ but not $F_n$ \cite[Corollary G]{BelkZaremsky22}. By Theorem \ref{thm:acyclic-t-br-thom} and \cite[Theorem 3.4]{BelkZaremsky22}, we therefore have the following.

\begin{cor}\label{cor:acyc-Fn-nFn+1-simple}
For any $n\geq 1$, there is a simple, acyclic group of type $F_n$ but not $F_{n+1}$.
\end{cor}

\begin{rem}
The first family of simple groups of type $F_n$ but not $F_{n+1}$ was constructed by Skipper, Witzel and Zaremsky \cite[Theorem 7.1]{SkipperWitzelZaremsky19} using the R\"over--Nekrashevych groups $\V_d(G)$ associated to certain self-similar subgroups $G$ of the automorphism group of the rooted $d$-ary tree (for $d \geq 4$). It is not clear whether these groups could also be acyclic, as the Higman--Thompson group $\V_d$ (as well as its index-$2$ subgroup when $d$ is odd) fails to be acyclic as soon as $d\geq 3$ \cite[\S 6]{SW19}. See also \cite{MillerSteinberg24} for more calculations of the homology of Röver--Nekrashevych groups.
\end{rem}

\begin{rem}
Recall that a group is called is \emph{universally boundedly acyclic} if it is boundedly acyclic over any complete valued field; see \cite[\S 5]{FFLM23} for more information. By \cite[Theorem 5.2]{FFLM23}, a group is universally boundedly acyclic if and only if it is boundedly acyclic over $\mathbb{R}$ and acyclic (over $\mathbb{Z}$). To the best of our knowledge, the only groups that were previously known to be universally boundedly acyclic were binate groups \cite[Theorem 1.4]{FFLM23} and Thompson's group $\V$ (by combining \cite{Konstantin22} and \cite{SW19}). Now, by Theorems \ref{thm:VG} and \ref{thm:acyclic-t-br-thom} combined with \cite[Theorems 0.1 and 0.11]{WuWuZhaoZhou2025}, we know that all labelled Thompson groups and all twisted Brin--Thompon groups are universally boundedly acyclic.
\end{rem}

\subsection*{Strategy and outline.}

Our proof of acyclicity both for labelled Thompson groups (Theorem \ref{thm:VG}) and for twisted Brin--Thompson groups (Theorem \ref{thm:acyclic-t-br-thom}) uses the framework of topological groupoids. We first build topological groupoids whose topological full groups are isomorphic to the groups in question, then apply Xin Li's results \cite{Li2024} to reduce the proof to calculating the homology of topological groupoids, which turns out to be much easier to handle. In \S\ref{section:top-groupoids} we review the basic properties of and results about topological groupoids that we need. We then prove our results in \S\ref{section:proofs-of-theorems}: Theorem \ref{thm:VG} in \S\ref{subsection:label-thom} (see Theorem \ref{thm:laV-acyc}) as well as Proposition \ref{aprop:VG-centre} (see Propositions \ref{prop:centre-VG} and \ref{prop:torsion-VG}), and Theorem \ref{thm:acyclic-t-br-thom} in \S\ref{subsection:twisted-BT} (see Theorem \ref{thm:tbt-acyc}).

In future work, we plan to adapt the methods of \cite{SW19}, rather than using topological groupoids, to calculate the homology of generalised Röver--Nekrashevych groups, including labelled Higman--Thompson groups $\V_d(G)$ for any $d\geq 2$.

\subsection*{Acknowledgments.}
MP was partially supported by grants of the Romanian Ministry of Research, Innovation and Digitization, CNCS - UEFISCDI, project numbers PN-IV-P1-PCE-2023-2001 and PN-IV-P2-2.1-TE-2023-2040, within PNCDI IV.

XW is currently a member of LMNS and is supported by NSFC No.12326601. He thanks Xin Li and Jianchao Wu for several inspiring discussions on topological groupoids. He thanks Matthew Zaremsky for asking the question of whether there are acyclic groups of type $F_n$ but not $F_{n+1}$ during the conference ``Topological and Homological Methods in Group Theory'' at Bielefeld in 2024. By either Corollary \ref{cor:acyc-Fn-nFn+1} or Corollary \ref{cor:acyc-Fn-nFn+1-simple}, the answer to his question is yes.

We thank Xin Li and Hiroki Matui for two wonderful lecture series on topological groupoids from which we benefited a lot. We also thank Jon Berrick, Francesco Fournier-Facio and Matthew Zaremsky for several helpful comments on a preliminary version of our paper.

\section{A quick introduction to topological groupoids}
\label{section:top-groupoids}

In this section we give a quick introduction to topological groupoids and their corresponding topological full groups; see \cite{Matui15,Renault80} for more information.

A groupoid is a small category whose morphisms are all invertible. As usual, we identify the groupoid with its set of morphisms, denoted by $\mathcal{G}$, and view its set of objects (also called units) $\mathcal{G}^{(0)}$ as a subset of $\mathcal{G}$ by identifying objects with the corresponding identity morphisms. By definition, a groupoid $\mathcal{G}$ comes equipped with range and source maps $r \colon \mathcal{G} \to \mathcal{G}^{(0)}$, $s \colon \mathcal{G} \to \mathcal{G}^{(0)}$, a multiplication map
\[
\mathcal{G} \tensor[_s]{\times}{_r} \mathcal{G} = \{ (g_1,g_2)\mid s(g_1) = r(g_2)\} \longrightarrow \mathcal{G}, \qquad (g_1,g_2) \longmapsto g_1g_2
\]
and an inversion map $g \mapsto g^{-1} \colon \cG \to \cG$ satisfying $r(g^{-1}) = s(g), s(g^{-1}) =r(g), gg^{-1} = r(g)$ and $g^{-1} g = s(g)$. These structure maps must satisfy the usual list of axioms so that $G$ is a small category; see \cite[Section 1.1]{Renault80} for more details.

\begin{defn}[Topological groupoids]
\label{def:top-gropuoid}
A \emph{topological groupoid} is a groupoid equipped with a topology making the composition and inversion maps continuous. In addition, we assume that the topology on $\cG$ makes its subspace $\cG^{(0)}$ (which we call the \emph{unit space}) locally compact and Hausdorff.
\end{defn}

We note that the topological groupoid $\cG$ itself need not be Hausdorff; only its subspace $\cG^{(0)}$ must be Hausdorff. However, in this paper we will only need to consider Hausdorff topological groupoids.

\begin{defn}[{\'E}tale groupoids]
A topological groupoid $\mathcal{G}$ is called \emph{étale} if the range and source maps are local homeomorphisms.
\end{defn}

We note that the definition implies that the unit space $\cG^{(0)}$ of an étale groupoid is an open subspace of the whole space $\cG$.

\begin{defn}
An open subspace $U \subseteq \mathcal{G}$ is called an \emph{open bisection} if the restricted range and source maps $r|_U \colon U\to r(U)$ and $s|_U \colon U\to s(U)$ are bijections (and hence homeomorphisms if $\cG$ is étale).
\end{defn}

If the groupoid $\mathcal{G}$ is étale, then $\mathcal{G}$ has a basis for its topology consisting of open bisections. We note that open bisections are always locally compact and Hausdorff because they are homeomorphic to open subspaces of the unit space.

\begin{defn}[Ample groupoids]
\label{def:ample}
A topological groupoid $\mathcal{G}$ is called \emph{ample} if it is étale and its unit space $\mathcal{G}^{(0)}$ is totally disconnected.
\end{defn}

Equivalently, an étale groupoid $\cG$ is ample if and only if it admits a basis for its topology consisting of compact open bisections. We also note that if $\cG$ is ample, then by definition its unit space $\mathcal{G}^{(0)}$ is assumed to be totally disconnected and locally compact, which implies that it admits a basis of compact open subsets.

Finally, we consider two notions of minimality for topological groupoids.

\begin{defn}
A topological groupoid is called \emph{minimal} if for each $x \in \mathcal{G}^{(0)}$, the corresponding orbit $\mathcal{G} \cdot x := \{r(g) \mid g\in s^{-1}(x)\}$ is dense in $\mathcal{G}^{(0)}$.
\end{defn}

\begin{defn}
\label{def:purely-inf-minimal}
An ample groupoid $\mathcal{G}$ is called \emph{purely infinite minimal} if for all compact open subspaces $U,V \subseteq \mathcal{G}^{(0)}$ with $V \neq \emptyset$, there exists a compact open bisection $\sigma \subseteq \mathcal{G}$ such that $s(\sigma) = U$ and $r(\sigma) \subseteq V$.
\end{defn}

As the terminology suggests, purely infinite minimal implies minimal.

\begin{lem}
An ample, purely infinite minimal groupoid is minimal.
\end{lem}
\begin{proof}
Denote the groupoid in question by $\cG$. Let $x \in \cG^{(0)}$ and let $V$ be a non-empty open subspace of $\cG^{(0)}$. We must show that $V \cap (\cG \cdot x) \neq \varnothing$, in other words that there exists $g \in \cG$ such that $s(g) = x$ and $r(g) \in V$. Since $\cG$ is ample, its unit space $\cG^{(0)}$ admits a basis of compact open subsets, as noted above, so we may assume that $V$ is compact, and we may also choose a compact open neighbourhood $U$ of $x \in \cG^{(0)}$. Applying Definition \ref{def:purely-inf-minimal}, we find a compact open bisection $\sigma \subseteq \cG$ with $s(\sigma) = U$ and $r(\sigma) \subseteq V$. Since $x \in U$, we may therefore find some $g \in \sigma$ with $s(g) = x$ and $r(g) \in V$, as required.
\end{proof}

\subsection{Examples.}

We now consider some important examples of topological groupoids. We first note that ample groupoids generalise discrete groups.

\begin{ex}
\label{ex:discrete-groups}
Every discrete group $G$, considered as a topological groupoid with a unique object and morphism space $G$, is an ample groupoid.
\end{ex}

The following groupoid will play a key role in our study.

\begin{ex}
\label{ex:SFT}
Consider $\{0,1\}^{\bN}$, the set of infinite sequences in $0,1$, equipped with the product topology. Note that $\{0,1\}^{\bN}$ is homeomorphic to the Cantor space; we will refer to it as the \emph{standard Cantor space} (often denoted by $\cC$). Define the one-sided shift $\rho \colon \{0,1\}^{\bN} \to \{0,1\}^{\bN}$ by sending $x_0x_1x_2\cdots$ to $x_1x_2\cdots$. The groupoid attached to this shift of finite type is given by 
\[
\cV_2 := \{ (y,n,x) \in \{0,1\}^{\bN} \times \bZ \times \{0,1\}^{\bN} \mid \exists~ l,m \in \bZ \text{ with } l,m\geq 0 \text{ such that } n=l-m \text{ and } \rho^l(x) = \rho^m(y) \}
\]

The topology of $\cV_2$ is generated by the collection of all sets of the form
\begin{equation}
\label{eq:basic-open-set-of-V2}
\{ (y,l-m,x)\in \cV_2 \mid x\in U, y\in V, \rho^l(x) = \rho^m(y)\}
\end{equation}
for fixed $l,m\in \bZ$ with $l,m\geq 0$, and $U,V$ open subspaces of $\{0,1\}^{\bN}$ such that $\rho^l$ and $\rho^m$ restrict to homeomorphisms
\[
V \xrightarrow[\;\cong\;]{\rho^l} \rho^l(V) = \rho^m(U) \xleftarrow[\;\cong\;]{\rho^m} U.
\]
We note for future reference that the underlying space of the topological groupoid $\cV_2$ splits as a disjoint union, indexed by $i \in \bZ$, of the open subspaces
\[
\cV_2[i] := \{ (y,n,x) \in \cV_2 \mid n=i \}.
\]

The unit space of $\cV_2$ is given by $\cV_2^{(0)} = \{(x,0,x) \in \cV_2 \mid x \in \{0,1\}^{\bN}\}$, which is the diagonal subspace of the product $\{0,1\}^{\bN} \times \{0,1\}^{\bN}$, hence canonically homeomorphic to $\{0,1\}^{\bN}$. The source and range maps are given by $s((y,n,x)) = x$, $r((y,n,x)) = y$, the multiplication is given by $(z,n',y)(y,n,x) = (z,n+n',x)$ and the inversion is given by $(y,n,x)^{-1} = (x,-n,y)$.

Note that the basic open set \eqref{eq:basic-open-set-of-V2} is naturally in bijection with $U$, and with $V$, under the source and range maps respectively. From this observation, it is clear that the source and range maps of $\cV_2$ are local homeomorphisms. Since the Cantor space $\cV_2^{(0)} \cong \{0,1\}^\bN$ is totally disconnected, the groupoid $\cV_2$ is ample.
\end{ex}

\begin{ex}
\label{ex:product}
Given any two topological groupoids $\cG_1,\cG_2$, their product $\cG_1\times \cG_2$ naturally forms a topological groupoid, with the topology of $\cG_1\times \cG_2$ given by the product topology. The unit space $(\cG_1\times \cG_2)^{(0)}$ is simply the product of the two unit spaces $\cG_1 ^{(0)}\times \cG_2^{(0)}$. If $\cG_1$ and $\cG_2$ are both étale (resp.\ ample), then $\cG_1\times \cG_2$ is also étale (resp.\ ample).
\end{ex}

\begin{ex}[{cf.\ \cite[Definition I.1.7]{Renault80}}]
\label{ex:semi-direct-product}
Let $\cG$ be a topological groupoid and let $H$ be a discrete group acting on $\cG$ from the left via a homomorphism $f \colon H \to \mathrm{Aut}(\cG)$. We can then form the semi-direct product $\cG \rtimes_f H$ on the product space $\cG \times H$ with the following groupoid structure: $(\gamma, h) $ and $ (\gamma', h')$ are composable if and only if $\gamma$ and $f(h)(\gamma')$ are composable, in which case the multiplication is given by
\[
(\gamma, h) \cdot (\gamma', h') = (\gamma \cdot f(h)(\gamma'),hh').
\]
Inversion is given by
\[
(\gamma,h)^{-1} = (f(h^{-1})(\gamma^{-1}),h^{-1})
\]
and the range and source maps are given by 
\[
r((\gamma,h)) = (r(\gamma),e) \quad\text{and}\quad s((\gamma,h)) = (f(h^{-1})(s(\gamma)),e),
\]
where $e$ denotes the identity element of $H$. Note that the unit space of $\cG \rtimes_f H$ is the subspace $\cG^{(0)} \times \{e\}$ of the product space $\cG \times H$, so it is homeomorphic to $\cG^{(0)}$; in particular it is locally compact and Hausdorff. Moreover, it follows directly from the definitions that if $\cG$ is an étale (resp.\ ample) groupoid, then $\cG \rtimes_f H$ is also étale (resp.\ ample).

Finally, we note that there is a natural homomorphism $\tilde{f} \colon \cG \rtimes_f H \to H$ given by $\tilde{f}(\gamma, h) = h$.
\end{ex}

\subsection{Topological full groups.}

We now consider the topological full group of a topological groupoid. To avoid technicalities that will not be relevant for this paper, we will restrict ourselves to ample groupoids with compact unit space.

\begin{defn}[Topological full group]
Let $\mathcal{G}$ be an ample groupoid with compact unit space $\mathcal{G}^{(0)}$. Its \emph{topological full group} $F(\mathcal{G})$ is the group of compact open bisections $\sigma \subseteq \mathcal{G}$ with $r(\sigma) = \mathcal{G}^{(0)} = s(\sigma)$. Multiplication in $F(\mathcal{G})$ is given by multiplication of bisections, i.e.\ $\sigma \tau = \{gh \mid g\in \sigma, h\in \tau, s(g) = r(h)\}$, the identity element is $\sigma = \mathcal{G}^{(0)}$ and inversion is given by $\sigma^{-1} = \{ g^{-1} \mid g \in \sigma \}$.
\end{defn}

\begin{defn}
\label{def:effective}
An ample groupoid $\cG$ with compact unit space is called \emph{effective} if the interior of the isotropy subgroupoid $\{g \in \cG\mid r(g) = s(g)\}$ coincides with $\cG^{(0)}$.
\end{defn}

\begin{rem}
Note that the interior of the isotropy subgroupoid always contains the unit space $\cG^{(0)}$; the effectivity condition says that there does not exist any strictly larger open subspace of $\cG$ inside the isotropy subgroupoid.
\end{rem}

If $\cG$ is an effective, ample, Hausdorff groupoid, then the homomorphism $F(\cG) \to \mathrm{Homeo}(\cG^{(0)})$ sending $\sigma \in F(\cG)$ to the homeomorphism $\cG^{(0)} \to \cG^{(0)}$ given by $s(g) \mapsto r(g)$ for $g \in \sigma$ is injective (see for example \cite[Lemma 3.1]{NylandOrtega2019}), so that we may view $F(\cG)$ naturally as a subgroup of $\mathrm{Homeo}(\cG^{(0)})$. We record this as a lemma for future use.

\begin{lem}
\label{lem:full-group-faithful}
If $\cG$ is an effective, ample, Hausdorff groupoid, then $F(\cG)$ is naturally isomorphic to a subgroup of $\mathrm{Homeo}(G^{(0)})$ via the embedding taking $\sigma$ to $r_U \circ (s|_U)^{-1}$.
\end{lem}

It is well-known that the topological full group $F(\cV_2)$ of the topological groupoid $\cV_2$ is the Thompson group $\V$. For completeness, we give a detailed proof of this in Lemma \ref{lem:groupdV2-full-V}, after properly introducing the Thompson group $\V$.

\subsection{Homology of topological full groups.}

Recall that the main task of this paper is to calculate the homology of some Thompson-like groups by viewing them as topological full groups. There is a homology theory for étale groupoids: see for example \cite{CrainicMoerdijk00} or \cite[\S 3]{Matui2012} for details. It turns out that the homology of étale groupoids is often much easier to calculate than the homology of their topological full groups. The following theorem of Xin Li allows us to exploit this by connecting the homology of a topological groupoid with the homology of its topological full group.

\begin{thm}[{\cite[Corollary D]{Li2024}}]
\label{thm:homgy-groupoid-fullgroup}
Let $\cG$ be an ample groupoid that is purely infinite minimal and whose unit space does not have isolated points.
Fix $k\in \bZ$ with $k > 0$. If $H_\ast(\cG) = \{0\}$ for all $\ast<k$, then $H_\ast(F(\cG)) = \{0\}$ for all $0<\ast<k$ and $H_k(F(\cG)) \cong H_k(\cG)$. In particular, if $H_\ast(\cG) = \{0\}$ for all $\ast \geq 0$, then $F(\cG)$ is acyclic.
\end{thm}

\section{Homology of some Thompson-like groups}
\label{section:proofs-of-theorems}

In this section we calculate the homology of labelled Thompson groups and of twisted Brin--Thompson groups by first realising them as topological full groups of some associated étale groupoids, then proving that the groupoid homology of these étale groupoids vanishes in all degrees, which implies that the topological full groups are acyclic by Theorem \ref{thm:homgy-groupoid-fullgroup}.

\subsection{Labelled Thompson groups}
\label{subsection:label-thom}

We first consider labelled Thompson groups. For a more detailed introduction to these groups, see for example \cite[\S 1]{WuWuZhaoZhou2025}; here we just discuss the minimum necessary for our purposes.

Let us first recall the definition of the labelled Thompson groups. Let $\{0,1\}^*$ denote the set of finite words in the two-element alphabet $\{0,1\}$ and note that the underlying set of the Cantor space $\{0,1\}^{\bN}$ may be thought of as the set of infinite ($\bN$-indexed) words in $\{0,1\}$. Given any finite word $w \in \{0,1\}^*$, let $\lvert w\rvert$ denote its word length and let $w\{0,1\}^{\bN}$ denote the subspace of the Cantor space $\{0,1\}^{\bN}$ consisting of (infinite) words starting with $w$. Note that $w\{0,1\}^{\bN}$ is a clopen subset. We shall refer to these clopen subsets of $\{0,1\}^{\bN}$ as \emph{standard clopen subsets}.

\begin{defn}
A finite subset $W \subset \{0,1\}^*$ is called a \emph{partition set} of $\{0,1\}^\bN$ if $w_1\{0,1\}^{\bN} \cap w_2\{0,1\}^{\bN} = \emptyset$ for any two distinct $w_1,w_2\in W$ and $\bigcup_{w\in W} w\{0,1\}^\bN = \{0,1\}^{\bN}$. In other words, the corresponding standard clopen subsets $w\{0,1\}^\bN$ for $w \in W$ must form a set-theoretic partition of $\{0,1\}^\bN$.
\end{defn}

Note that there is a lexicographical ordering on $\{0,1\}^*$ induced by $0<1$. We will always list the elements of a partition set from smallest to largest acccording to this ordering.

\subsubsection*{The Thompson group.}
Recall now that the Thompson group $\V$ is the subgroup of the homeomorphism group of the Cantor space $\{0,1\}^\bN$ consisting of those homeomorphisms $\phi$ for which there exists a partition set $W$ and a map $\sigma \colon W \to \{0,1\}^*$ such that $\phi$ is given by $\phi(wx) = \sigma(w)x$ for any $w\in W$ and $x\in \{0,1\}^{\bN}$.

If we write $W' = \sigma(W)$, then $W'$ is also a partition set and $\sigma$ is a bijection between the partition sets $W$ and $W'$; thus we can represent the homeomorphism $\phi$ by the triple $(W',\sigma,W)$, which we call a \emph{table}. However, such a representation is not unique as one can replace an element $u\in W$ by the pair of elements $u0$ and $u1$, correspondingly $\sigma(u) \in W'$ by $\sigma(u)0$ and $\sigma(u)1$, without changing the homeomorphism $\phi$. Let us denote these new partition sets by $\bar{W}$ and $\bar{W}'$ and the new bijection by $\bar{\sigma}$. Thus $\bar{\sigma}(w) = \sigma(w)$ if $w \neq u$, and $\bar{\sigma}(u0) = \sigma(u)0$ and $\bar{\sigma}(u1) = \sigma(u)1$. We call the table $(\bar{W}',\bar{\sigma},\bar{W})$ obtained in this way an \emph{expansion} $(W',\sigma, W)$ and we call $(W',\sigma, W)$ a \emph{reduction} of $(\bar{W}',\bar{\sigma},\bar{W})$. We denote the equivalence class of the table $(W',\sigma, W)$ under reduction and expansion by $[W',\sigma,W]$.

Each element of $V$ can be uniquely represented by an equivalence class of tables as in the paragraph above. Moreover, since we always list the elements of the partition sets $W$ and $W'$ from smallest to largest, we may identity both of them with the set $[n]= \{1,2,\ldots,n\}$ via the unique order-preserving bijection. Thus we may assume that $\sigma$ lies in the permutation group $\mathrm{S}_n$ of $[n]$.

\subsubsection*{Labelled Thompson groups.}
Now let us fix a (discrete) group $G$ and define the labelled Thompson group $\V(G)$. The elements of $\V(G)$ will be equivalence classes of $G$-tables under the equivalence relation generated by $G$-expansion and $G$-reduction; to define the underlying set of $\V(G)$ we therefore just have to define $G$-tables, $G$-expansion and $G$-reduction.

A \emph{$G$-table} is a triple of the form $(\{w'_1,\ldots, w'_n\} ,((g_1,\ldots,g_n),\sigma),\{w_1,\ldots,w_n\})$ where $\{w_1,\ldots,w_n\}$ and $\{w'_1,\ldots, w'_n\}$ are partition sets whose elements are listed in lexicographical order, $\sigma$ is an element in $\mathrm{S}_n$ and each $g_i \in G$. We will freely pass between viewing $\sigma$ as an element of $\mathrm{S}_n$ and as a bijection $\{w_1,\ldots,w_n\} \to \{w'_1,\ldots, w'_n\}$ via the unique order-preserving bijections with $[n]$, as above. When we are viewing $\sigma$ as an element of $\mathrm{S}_n$, we may also view the tuple $((g_1,\ldots,g_n),\sigma)$ as an element of the wreath product $G \wr \mathrm{S}_n$, so we may write any $G$-table more succinctly as $(W',\alpha,W)$ for $\alpha \in G \wr \mathrm{S}_n$ and partition sets $W,W'$.

The \emph{$G$-expansion} of $(\{w'_1,\ldots, w'_n\} ,((g_1,\ldots,g_n),\sigma),\{w_1,\ldots,w_n\})$ at the position $i$ is defined by:
\[
(\{w_1',\ldots,w_{\sigma(i)}'0, w_{\sigma(i)}'1,\ldots, w_n'\} ,((g_1,\ldots,g_{i-1},g_i,g_i,g_{i+1},\ldots,g_n),\bar\sigma),\{w_1,\ldots,w_i0,w_i1,\ldots,w_n\})
\]
where $\bar{\sigma}(w_j) = \sigma(w_j)$ for $j\neq i$ and $\bar{\sigma}(w_i 0) = \sigma(w_i)0$ and $\bar{\sigma}(w_i 1) = \sigma(w_i)1$. In general, a \emph{$G$-expansion} of a $G$-table is a $G$-expansion as described above at any position $1\leq i\leq n$, and a \emph{$G$-reduction} is the reverse operation. We denote the equivalence class of the $G$-table $(W',\alpha,W)$ under $G$-expansion and $G$-reduction by $[W',\alpha,W]$. As indicated above, these equivalence classes are the elements of $\V(G)$.

We define the group operation of $\V(G)$ as follows. First, we define
\[
[U,\beta,W'] [W',\alpha, W] = [U,\beta\alpha,W],
\]
where $\alpha,\beta \in G \wr \mathrm{S}_n$ (here, $n$ is the common size of the partition sets $W,W',U$) and $\beta\alpha$ uses the multiplication rule in the wreath product $G\wr \mathrm{S}_n$. To verify that this gives a well-defined operation on $\V(G)$, one checks that (i) this definition is stable under $G$-expansion and $G$-reduction and (ii) for any two $G$-tables $(U',\beta,U)$ and $(W',\alpha,W)$ one may apply $G$-expansion operations until one has $U=W'$. The identity element is $(\{\varnothing\},\mathrm{id},\{\varnothing\})$, where $\varnothing$ denotes the empty word, and the inverse of $(W',\alpha,W)$ is given by $(W,\alpha^{-1},W')$, where $\alpha^{-1}$ is computed using the inverse in the wreath product $G \wr \mathrm{S}_n$.

Finally, we note that there is a natural embedding of groups $\iota_0 \colon G \hookrightarrow \V(G)$ given by sending $g \in G$ to $[\{0,1\}, ((g,e),\mathrm{id}),\{0,1\}]$, where $e$ denotes the identity element of $G$.

\subsubsection*{Labelled Thompson groups as topological full groups.}

We now build a topological groupoid whose topological full group is isomorphic to $\V(G)$. Recall that any discrete group may be viewed as a groupoid with a unique object and the elements of the group as morphisms. Equipping the morphism space with the discrete topology, we can therefore view any discrete group $G$ as a topological groupoid. Using this viewpoint, we can form the product groupoid $\mathcal{V}_2 \times G$, where $\cV_2$ is the topological groupoid associated to the Thompson group $\V$ described in Example \ref{ex:SFT}.

\begin{prop}
\label{prop:full-group-VH}
The topological full group of the topological groupoid $\mathcal{V}_2 \times G$ is isomorphic to $\V(G)$.
\end{prop}

To prove this, it will be convenient to see first why the topological full group $F(\mathcal{V}_2)$ is isomorphic to the Thompson group $\V$. We provide a proof of this standard fact for the reader's convenience.

\begin{lem}
\label{lem:groupdV2-full-V}
We have $F(\mathcal{V}_2) \cong \V$. 
\end{lem}

\begin{proof}
Recall from Example \ref{ex:SFT} that $\cV_2$ is an ample groupoid with unit space the Cantor space $\{0,1\}^\bN$. We would like to view its topological full group as a subgroup of $\Homeo(\{0,1\}^\bN)$ by Lemma \ref{lem:full-group-faithful}, so we first check that it is effective, i.e., that the interior of its isotropy subgroupoid
\[
\mathrm{Iso}(\cV_2) = \{(y,n,x) \in \cV_2 \mid r(y,n,x) = s(y,n,x)\}
\]
coincides with its unit space $\cV_2^{(0)}$. Since $r((y,n,x)) = y$ and $s((y,n,x)) = x$, the elements in the isotropy subgroupoid are precisely those of the form $(x,n,x)$. Since $\cV_2$ splits as the disjoint union of clopen subspaces $\cV_2[i]$ for $i \in \bZ$ (see Example \ref{ex:SFT}), the isotropy subgroupoid splits correspondingly into clopen subspaces $\mathrm{Iso}(\cV_2)[i] = \{ (y,n,x) \in \cV_2 \mid x=y \text{ and } n=i \}$.

Clearly we have $\mathrm{Iso}(\cV_2)[0] = \cV_2^{(0)}$, so what we have to show is that the subspace $\bigcup_{i\neq 0} \mathrm{Iso}(\cV_2)[i]$ of $\cV_2$ has empty interior. To see this, we first observe that every non-empty open subset of $\cV_2$ is uncountable; this follows from the fact that it is locally homeomorphic to the Cantor space $\{0,1\}^\bN \cong \cV_2^{(0)}$ and the fact that the Cantor space has this property. Second, we observe that the elements of $\bigcup_{i\neq 0} \mathrm{Iso}(\cV_2)[i]$ are precisely those of the form $(x,i,x)$ for $i \neq 0$, which implies that $x$ must be eventually $i$-periodic. There are only countably many eventually periodic sequences, so the subspace $\bigcup_{i\neq 0} \mathrm{Iso}(\cV_2)[i]$ must be countable. Hence, by the first observation, the only open subset contained in it is the empty set. Thus, it has empty interior, and we have shown that $\cV_2$ is effective.

It is also clear that $\cV_2$ is Hausdorff (it is a subspace of the product $\{0,1\}^\bN \times \bZ \times \{0,1\}^\bN$), so Lemma \ref{lem:full-group-faithful} tells us that $F(\mathcal{V}_2)$ is a subgroup of 
$\Homeo(\{0,1\}^\bN)$. By definition, the Thompson group $\V$ is also a subgroup of $\Homeo(\{0,1\}^\bN)$. We will show that they are equal as subgroups.

\textbf{1.} $F(\mathcal{V}_2) \supseteq \V$.
As discussed above, given any element $g\in \V$, we can represent it by a table $(W',\sigma,W)$ where $W,W'$ are finite partition sets of $\{0,1\}^\bN$. Suppose that $w\in W$ is mapped to $w'\in W'$ under $\sigma$. We need to prove that the prefix substitution homeomorphism $w\{0,1\}^{\bN} \to w'\{0,1\}^{\bN}$ can be represented by a compact open bisection in $\cV_2$. This compact open bisection may be represented diagrammatically as follows:
\[
w'\{0,1\}^{\bN}  \xrightarrow[\cong]{\;\rho^{\lvert w' \rvert}\;} \{0,1\}^{\bN} \xleftarrow[\cong]{\;\rho^{\lvert w \rvert}\;} w\{0,1\}^{\bN},
\]
where $\lvert - \rvert$ denotes word length. Explicitly, the compact open bisection is:
\[
\{(y , \lvert w \rvert - \lvert w' \rvert , x) \in \cV_2 \mid y\in w'\{0,1\}^{\bN}, x\in w\{0,1\}^{\bN},  \rho^{\lvert w' \rvert}(y) =\rho^{\lvert w \rvert}(x)\} ,
\]
which may alternatively be written as
\[
\{ (w'z , \lvert w \rvert - \lvert w' \rvert , wz) \mid z \in \{0,1\}^\bN \} .
\]

\textbf{2.} $F(\mathcal{V}_2) \subseteq \V$.
Note first that the unit space of $\cV_2$ is the Cantor space, which is compact. Now, any element $f \in F(\mathcal{V}_2)$ may be represented as the disjoint union of finitely many compact open bisections of the form
\[
\{(y,l-m,x) \in \cV_2 \mid y\in U', x\in U,\rho^m(y) = \rho^{l}(x)\},
\]
where $l,m \in \bZ$ and $U,U'$ are open subspaces of $\{0,1\}^{\bN}$. Up to subdivision, we may assume that $U$ and $U'$ are standard clopen subsets. Thus we may represent $f$ as the disjoint union of finitely many compact open bisections of the form
\[
\{(y , \lvert w_i \rvert - \lvert w'_i \rvert , x) \in \cV_2 \mid y\in w'_i\{0,1\}^{\bN}, x\in w_i\{0,1\}^{\bN},\rho^{\lvert w_i' \rvert}(y) = \rho^{\lvert w_i \rvert}(x)\},
\]
where $w_i$ and $w'_i$ are words in $\{0,1\}^{\ast}$. Putting all of the $w_i$ (resp.\ $w'_i$) together, we get a partition set $W$ (resp.\ $W'$) of $\{0,1\}^{\bN}$. Letting $\sigma$ be the bijection taking $w_i$ to $w'_i$, we have represented the homeomorphism $f$ by the table $(W',\sigma,W)$. But this, by definition, represents an element of the Thompson group $\V$.
\end{proof}

\begin{proof}[Proof of Proposition \ref{prop:full-group-VH}.]
We will construct group homomorphisms $I \colon F(\cV_2\times G) \to \V(G)$ and $J \colon \V(G) \to F(\cV_2\times G)$ that are inverse to each other.

Note first that the unit space of $\cV_2 \times G$ is $\cV_2^{(0)} \times \{e\}$, since the unit space of $G$ consists of just the identity element $e$ of the group (cf.\ also Example \ref{ex:product}). This is homeomorphic to the Cantor space $\cC$; in particular it is compact.

Any element $f\in F(\cV_2\times G)$ may be written as a disjoint union of finitely many compact open bisections of the form
\[
\{ ((y,l-m,x),g) \in \cV_2 \times G \mid x\in U, y\in U', \rho^{m}(y) = \rho^l(x) \},
\]
where $l,m\in \bZ$, $U,U'$ are open subspaces of $\{0,1\}^{\bN}$ and $g\in G$ is fixed. Just as before, we may assume by subdivision that $U,U'$ are standard clopen sets. In particular, we may represent the element $f$ by a $G$-table $(W',((g_1,\ldots,g_n),\sigma),W)$, where $W,W'$ are finite partition sets with $n$ elements. This gives an element in $\V(G)$, which we denote by $I(f)$. This defines the map $I$, which one may readily check is a group homomorphism.

To construct $J$, we start with an element of $\V(G)$, which is an equivalence class of $G$-tables, which we denote by $[W',((g_1,\ldots,g_n),\sigma),W]$. Suppose that $w\in W$ is mapped to $w'\in W'$ under $\sigma$, and that it is the $i$th element ($1\leq i\leq n$) in $W$ in the lexicographical ordering. We consider the compact open bisection represented by the following diagram:
\[
w'\{0,1\}^{\bN} \times \{e\} \xrightarrow[\cong]{\;\rho^{\lvert w' \rvert} \times k_i\;} \{0,1\}^{\bN} \times \{e\} \xleftarrow[\cong]{\;\rho^{\lvert w \rvert} \times h_i\;} w\{0,1\}^{\bN} \times \{e\},
\]
where $h_i,k_i$ are any two fixed elements in $G$ such that $k_i^{-1}h_i = g_i$. Note that here the choices of $h_i$ and $k_i$ are not unique. Taking the disjoint union of these compact open bisections over all $w \in W$, we obtain an element of $F(\cV_2\times G)$, which we define to be the image of $[W',((g_1,\ldots,g_n),\sigma),W]$ under $J$. This defines the map $J$, which one may again readily check is a group homomorphism.

One may directly check, unwinding the definitions, that $I\circ J = \mathrm{id}_{\V(G)}$ and $J\circ I = \mathrm{id}_{F(\cV_2\times G)}$.
\end{proof}

\begin{lem}
\label{lem:homology-prod-VH}
We have $H_i(\mathcal{V}_2 \times G) = 0$ for any $i\geq 0$.
\end{lem}
\begin{proof}
Recall that Matui proved in \cite[Theorem 4.14]{Matui2012} that $H_i(\mathcal{V}_2) = 0$ for any $i\geq 0$. The lemma now follows immediately from the K\"unneth formula \cite[Theorem 2.4]{Matui2016}.
\end{proof}

\begin{lem}\label{lem:minimal+comprision-VH}
The groupoid $\mathcal{V}_2 \times G$ is purely infinite minimal.
\end{lem}
\begin{proof}
Recall that the unit space of $G$ consists of the single identity element $e$. So the unit space of $\cV_2\times G$ is $\{0,1\}^{\bN} \times \{e\}$, which is a Cantor space. The result now follows from the fact that the topological groupoid $\mathcal{V}_2$ is purely infinite minimal. In fact, give any compact open subspaces $U,U' \subseteq ({\cV}_2\times G)^{(0)}$ with $U'\neq \emptyset$, we can find coverings of $U$ and $U'$ by disjoint standard clopen sets:
\[
U = \bigcup_{i\in I} u_i\{0,1\}^{\bN} \times \{e\} \quad\text{and}\quad U' = \bigcup_{j\in J} u'_j\{0,1\}^{\bN} \times \{e\},
\]
where $u_i, u'_j \in \{0,1\}^\ast$. Since $U$ and $U'$ are compact, we may take $I$ and $J$ to be finite. Up to decomposing $u'_j\{0,1\}^{\bN} \times \{e\}$ further into smaller standard clopen sets, we can assume that $\lvert I \rvert < \lvert J \rvert$. Then for each $1\leq i\leq \lvert I \rvert$, we can build a bisection as follows:
\[
u'_i\{0,1\}^{\bN} \times \{e\} \xrightarrow[\cong]{\;\rho^{\lvert u'_i \rvert} \times e\;} \{0,1\}^{\bN} \times \{e\} \xleftarrow[\cong]{\;\rho^{\lvert u_i \rvert} \times e\;} u_i\{0,1\}^{\bN} \times \{e\}.
\]
To check that this bisection is compact and open, we note that it is the compact open set described by:
\[
\{((y, \lvert u_i \rvert - \lvert u'_i \rvert , x),e ) \mid x\in u_i\{0,1\}^{\bN}, y \in u'_i\{0,1\}^{\bN}, \rho^{\lvert u'_i \rvert}(y) = \rho^{\lvert u_i \rvert}(x) \} ,
\]
or equivalently:
\[
\{ ((u'_i z , \lvert u_i \rvert - \lvert u'_i \rvert , u_i z) , e) \mid z \in \{0,1\}^\bN \} .
\]

Thus, for each $1\leq i\leq \lvert I \rvert$, we have a compact open bisection taking $u_i\{0,1\}^{\bN}$ onto $u'_i\{0,1\}^{\bN}$. Since they have disjoint ranges and sources, we may take their union to obtain a compact open bisection taking $U = \bigcup_{1\leq i\leq |I|} u_i\{0,1\}^{\bN}$ onto $\bigcup_{1\leq i\leq |I|} u'_i\{0,1\}^{\bN} \subseteq U'$.
\end{proof}

We are now ready to prove the following theorem, which is Theorem \ref{thm:VG} of the introduction.

\begin{thm}
\label{thm:laV-acyc}
The labelled Thompson group $\V(G)$ is acyclic for any discrete group $G$.
\end{thm}
\begin{proof}
Note first that the groupoid $\mathcal{V}_2 \times G$ is an ample groupoid whose unit space is the Cantor set (see Examples \ref{ex:discrete-groups}, \ref{ex:SFT} and \ref{ex:product}); in particular, its unit space does not have isolated points. By Lemma \ref{lem:minimal+comprision-VH}, the groupoid $\cV_2\times G$ is purely infinite minimal. Thus, by Lemma \ref{lem:homology-prod-VH} and Theorem \ref{thm:homgy-groupoid-fullgroup}, the topological full group $F(\cV_2 \times G)$ is acyclic, and the result now follows from Proposition \ref{prop:full-group-VH}.
\end{proof}

\subsubsection*{The centre of labelled Thompson groups.}
We also take this opportunity to calculate the centre of $\V(G)$. To do this, we will use the following result of Wu-Wu-Zhao-Zhou \cite{WuWuZhaoZhou2025}. 

\begin{thm}[{\cite[Theorem 2.1]{WuWuZhaoZhou2025}}]
\label{thm:VG-normal-subgroup}
Any proper normal subgroup of $\V(G)$ is contained in the kernel of the map $\pi \colon \V(G) \to \V$ that forgets the labels, i.e., it sends $[W',((g_1,\ldots, g_n),\sigma), W] \in \V(G)$ to $[W',\sigma, W] \in \V$.
\end{thm}

Let us now consider the embedding $\iota_\varnothing \colon G \hookrightarrow \V(G)$ given by sending $g\in G$ to $[\{\varnothing\},((g),\mathrm{id}),\{\varnothing\}]$, where $\varnothing$ denotes the empty word. Let $Z(H)$ denote the centre of a group $H$.

\begin{prop}
\label{prop:centre-VG}
The centre of $\V(G)$ is $Z(\iota_\varnothing(G))$, in particular it is isomorphic to $Z(G)$.
\end{prop}
\begin{proof}
Let us first show that $Z(\iota_\varnothing(G))\subseteq Z(\V(G))$. Let $[W', \alpha, W]$ be an element in $\V(G)$ where $W,W'$ are partition sets and $\alpha = ((g_1,\ldots,g_n), \sigma)\in G\wr \mathrm{S}_n$. For any $z\in Z(G)$, we have: 
\begin{align*}
[\{\varnothing\},((z),\mathrm{id}),\{\varnothing\}] [W', \alpha, W] [\{\varnothing\}, & ((z),\mathrm{id}),\{\varnothing\}]^{-1} = \\
       &=[W', ((z,\ldots, z),\mathrm{id})\alpha ((z^{-1},\ldots,z^{-1}),\mathrm{id}) ,W] \\
       &= [W', ((z,\ldots, z),\mathrm{id}) ((g_1,\ldots,g_n), \sigma)  ((z^{-1},\ldots,z^{-1}),\mathrm{id}) ,W] \\
       &=[W', ((z,\ldots, z),\mathrm{id}) ((g_1z^{-1},\ldots,g_nz^{-1}), \sigma),W]  \\
       & = [W', ((zg_1z^{-1},\ldots,zg_nz^{-1}), \sigma),W]  \\
       & = [W', ((g_1,\ldots,g_n), \sigma),W] ~~~(\text{since } z \in Z(G)) \\
       & = [W', \alpha, W].
\end{align*}
Thus $\iota_\varnothing(z) = [\{\varnothing\},((z),\mathrm{id}),\{\varnothing\}]$ lies in the centre of $\V(G)$.

We proceed to show that $Z(\iota_\varnothing(G))\supseteq Z(\V(G))$. Let $\pi \colon \V(G) \to \V$ be the homomorphism that forgets all the labels in $G$. By Theorem \ref{thm:VG-normal-subgroup}, $Z(\V(G))$ is a subgroup of $\mathrm{ker}{(\pi)}$. Hence any element $\zeta \in Z(\V(G))$ may be represented by a $G$-table of the form $(P_n,((z_1,z_2,\ldots, z_{2^n}),\mathrm{id}),P_n)$ for some $n\geq 1$ and some elements $z_i \in G$, where $P_n$ denotes the partition set of $\{0,1\}^\bN$ given by the set of all words of length exactly $n$. Note that $\lvert P_n \rvert = 2^n$. Let $G_n$ be the subgroup $\{[P_n,((g_1,g_2,\ldots, g_{2^n}),\mathrm{id}),P_n] \mid g_i\in G,1\leq i\leq 2^n \}$ of $\V(G)$. Then $G_n \cong G^{2^n}$ and the centre of $G_n$ is $(Z(G))^{2^n}$ under this identification. Since $\zeta$ lies in the centre of $\V(G)$, it commutes with every element of $G_n$. This implies that $z_i\in Z(G)$ for each $1\leq i\leq 2^n$. We next verify that $z_i = z_j$ for all $1\leq i,j\leq 2^n$. If this is not the case, we choose some $i\neq j$ such that $z_i \neq z_j$ and let $\sigma$ be the bijection of $P_n$ that transposes $i$ and $j$ and is the identity elsewhere. Then $[P_n,((e,e,\ldots, e),\sigma),P_n]$ conjugates $\zeta = [P_n,((z_1,z_2,\ldots, z_{2^n}),\mathrm{id}),P_n]$ to $[P_n,((z_{\sigma(1)},z_{\sigma(2)},\ldots, z_{\sigma(2^n)}),\mathrm{id}),P_n]$. Since $z_i\ne z_j$, we have
\[
[P_n,((z_1,z_2,\ldots, z_{2^n}),\mathrm{id}),P_n] \neq [P_n,((z_{\sigma(1)},z_{\sigma(2)},\ldots, z_{\sigma(2^n)}),\mathrm{id}),P_n]
\]
by the definition of $\sigma$, contradicting the assumption that $\zeta$ is in the centre of $\V(G)$. Thus we have shown that every element of the centre of $\V(G)$ is of the form $[P_n,((z,z,\ldots, z),\mathrm{id}),P_n]$ for some $z\in Z(G)$. By definition, $[P_n,((z,z,\ldots, z),\mathrm{id}),P_n] = [\{\varnothing\},((z),\mathrm{id}),\{\varnothing\}] = \iota_\varnothing(z)$. This completes the proof.
\end{proof}

\subsubsection*{Strong torsion generation.}
Finally, we observe that $\V(G)$ is strongly torsion generated.

\begin{prop}
\label{prop:torsion-VG}
For each $n\geq 2$, there is an element $g \in \V(G)$ of order $n$ that normally generates $\V(G)$. In other words, $\V(G)$ is strongly torsion generated.
\end{prop}
\begin{proof}
Let $g \in \V$ be an element of order $n$, and consider it as an element of $\V(G)$ by choosing all labels to be the identity element of $G$. We must show that its normal closure $\langle g \rangle^{\V(G)}$ is all of $\V(G)$. But this follows from Theorem \ref{thm:VG-normal-subgroup}, since $\langle g \rangle^{\V(G)}$ does not lie in the kernel of $\pi \colon \V(G) \to \V$.
\end{proof}

\subsection{Twisted Brin--Thompson groups}
\label{subsection:twisted-BT}

In this subsection, we construct topological groupoids whose corresponding topological full groups are the twisted Brin--Thompson groups. We then apply Theorem \ref{thm:homgy-groupoid-fullgroup} to show that all twisted Brin--Thompson groups are acyclic. We begin with a review of the basics of twisted Brin--Thompson groups; for further details, see \cite{BelkZaremsky22}.

\subsubsection*{Brin--Thompson groups.}

Recall that $\cC = \{0,1\}^{\bN}$ is the the space of infinite sequences in the alphabet $\{0,1\}$ equipped with the product topology. For any set $S$, its associated \emph{Cantor cube} $\cC^{S}$ is the space $\prod_{s \in S} \cC$, again equipped with the product topology, whose elements are all functions $S \rightarrow \cC$. Recall that $\{0,1\}^\ast$ denotes the set of finite words in the alphabet $\{0,1\}$. We say that a function $\psi \colon S \rightarrow \{0,1\}^{\ast}$ has \emph{finite support} if $\psi(s) \in \{0,1\}^{\ast}$ is the empty word for all but finitely many $s \in S$. Given such a function $\psi$, its associated \emph{dyadic brick} in $\cC^S$ is the subset
\[
B(\psi) = \left\{ \kappa \in \cC^{S} \mid \psi(s) \text { is a prefix of } \kappa(s) \text { for each } s \in S \right\} .
\]
For example, $\cC^S$ itself is the dyadic brick associated to the function $S \to \{0,1\}^\ast$ that sends every element to the empty word. Note that there is a canonical homeomorphism $h_{\psi} \colon \cC^{S} \to B(\psi)$ defined by
\[
h_{\psi}(\kappa)(s) = \psi(s) \cdot \kappa(s),
\]
where $\cdot$ denotes concatenation of words. More generally, if $B(\varphi)$ and $B(\psi)$ are dyadic bricks, we refer to the composition $h_{\psi} \circ h_{\varphi}^{-1}$ as the \emph{canonical homeomorphism} from $B(\varphi)$ to $B(\psi)$.

Given any two partitions $B(\varphi_{1}), \ldots, B(\varphi_{n})$ and $B(\psi_{1}), \ldots, B(\psi_{n})$ of $\cC^{S}$ into the same number of dyadic bricks, we can define a homeomorphism $\cC^{S} \rightarrow \cC^{S}$ by sending each $B(\varphi_{i})$ to the corresponding $B(\psi_{i})$ by the canonical homeomorphism. The group of all homeomorphisms of $\cC^{S}$ defined in this way is called the \emph{Brin--Thompson group $S\mathrm{V}$}. When $S$ is finite, this group was first defined by Brin \cite{Brin04}.

\subsubsection*{Twisted Brin--Thompson groups.}

Now let $G$ be a group acting faithfully on a countable set $S$. For each element $\gamma \in G$, let $\tau_{\gamma} \colon \cC^{S} \rightarrow \cC^{S}$ be the \emph{twist homeomorphism} that permutes the coordinates of points in $\cC^{S}$ according to $\gamma$. In other words, $\tau_{\gamma}$ is the homeomorphism of $\cC^{S}$ defined by
\begin{equation}
\label{eq:tau-gamma}
\tau_{\gamma}(\kappa)(s) = \kappa (\gamma^{-1} s)
\end{equation}
for all $\kappa \in \cC^{S}$ and $s \in S$. 

In general, given any two dyadic bricks $B(\varphi)$ and $B(\psi)$ and an element $\gamma \in G$, we define the associated twist homeomorphism $B(\varphi) \to B(\psi)$ to be the composition $h_{\psi} \circ \tau_{\gamma} \circ h_{\varphi}^{-1}$ (writing composition as usual from right to left), where $h_{\varphi} \colon \cC^{S} \to B(\varphi)$ and $h_{\psi} \colon \cC^{S} \to B(\psi)$ are the canonical homeomorphisms. We call this the \emph{canonical twist homeomorphism} associated to $\gamma$ between these dyadic bricks.

Given any two partitions $\{B(\varphi_{1}), \ldots, B(\varphi_{n})\}$ and $\{B(\psi_{1}), \ldots, B(\psi_{n})\}$ of $\cC^{S}$ into the same number of dyadic bricks, together with elements $\gamma_{1}, \ldots, \gamma_{n} \in G$, we can define a homeomorphism of $\cC^{S}$ by mapping each $B(\varphi_{i})$ to $B(\psi_{i})$ via the canonical twist homeomorphism associated to $\gamma_{i}$. For this homeomorphism, $\{B(\varphi_{1}), \ldots, B(\varphi_{n})\}$ is called the \emph{domain partition} and $\{B(\psi_{1}), \ldots, B(\psi_{n})\}$ is called the \emph{image partition}. The group of all homeomorphisms of $\cC^{S}$ constructed in this way forms the \emph{twisted Brin--Thompson group} $S\mathrm{V}_G$. The key property of $S\mathrm{V}_G$ is that it is always simple \cite[Theorem 3.4]{BelkZaremsky22}.

\subsubsection*{Topological groupoids associated to Brin--Thompson groups.}

Recall from Example \ref{ex:SFT} and Lemma \ref{lem:groupdV2-full-V} that $\mathcal{V}_2$ is the topological groupoid associated to the Thompson group $\V$. We now define a topological groupoid associated to the Brin--Thompson group $S\V$.

\begin{defn}
\label{defn:groupoid-svg}
Let $S\mathcal{V}_2$ be the topological groupoid with:
\begin{enumerate}
\item Unit space: $\prod_S \mathcal{V}_2^{(0)}$ equipped with the product topology.
\item Morphism space: $\left\lbrace (\gamma_s)_{s\in S} \in \prod_S \mathcal{V}_2 \bigm| \gamma_s \in \mathcal{V}^{(0)}_2 \text{ for all but finitely many } s \right\rbrace$ equipped with the subspace topology induced from the product topology on $\prod_S \cV_2$. A basis for this topology is given by $\left\lbrace \prod_{s \in S} U_s \bigm| U_s \subseteq \cV_2 \text{ is open and } U_s = \mathcal{V}^{(0)}_2 \text{ for all but finitely many } s \right\rbrace$.
\end{enumerate}
We note that, since $\cV_2^{(0)} \cong \cC$ and $S$ is countable, the unit space of $S\cV_2$ is homeomorphic to $\cC^S$, which is homeomorphic to $\cC$, so it is locally compact and Hausdorff. Thus $S\cV_2$ is indeed a topological groupoid according to our convention in Definition \ref{def:top-gropuoid}.
\end{defn}

If $S$ is finite, this is just the product groupoid $\prod_S \cV_2$, whereas if $S$ is infinite, it is a proper subgroupoid of $\prod_S \cV_2$. Let us denote the range and source maps of the copy of $\cV_2$ in the $s$ coordinate by $r_s$ and $s_s$. Then the range and source maps $r$ and $s$ of $S\cV_2$ are given by:
\[
r = (r_s)_{s\in S} \quad\text{and}\quad s = (s_s)_{s\in S}.
\]
Since $\cV_2$ is étale, the maps $r_s$ and $s_s$ are local homeomorphisms for each $s \in S$. Using this, we may deduce the following.

\begin{lem}
\label{lem:SV2-etale}
The topological groupoid $S\cV_2$ is ample, in particular it is étale.
\end{lem}
\begin{proof}
We must check that the range and source maps of $S\cV_2$ are local homeomorphisms, and that its unit space is totally disconnected. For the second condition, we simply recall from above that its unit space is homeomorphic to the Cantor space $\cC$ (since any countable product of copies of $\cC$ is homeomorphic to $\cC$), which is totally disconnected.

For the first condition, we fix a point $\gamma = (\gamma_s)_{s\in S}$ in $S\cV_2$ and must find an open neighbourhood $U$ of $\gamma$ in $S\cV_2$ such that $r(U)$ and $s(U)$ are open in $(S\cV_2)^{(0)}$ and the maps $r|_U \colon U \to r(U)$ and $s|_U \colon U \to s(U)$ are homeomorphisms. There are only finitely many coordinates $s \in S$ for which $\gamma_s \notin \cV_2^{(0)}$; for these values of $s$ let us choose any open neighbourhood $U_s$ of $\gamma_s$ in $\cV_2$ such that $r_s(U_s)$ and $s_s(U_s)$ are open in $\cV_2^{(0)}$ and the maps $r_s|_{U_s} \colon U_s \to r_s(U_s)$ and $s_s|_{U_s} \colon U_s \to s_s(U_s)$ are homeomorphisms; such neighbourhoods exist since we know that $r_s$ and $s_s$ are local homeomorphisms. For the other coordinates $s \in S$, where we have $\gamma_s \in \cV_2^{(0)}$, let us choose $U_s = \cV_2^{(0)}$. In this case it is also true that the restrictions of $r_s$ and $s_s$ to $U_s$ are homeomorphisms onto their images, since the source and range maps of any topological groupoid, restricted to the unit space, are simply the identity map of the unit space. In particular we have $r_s(U_s) = s_s(U_s) = \cV_2^{(0)}$ in this case.

We now take $U = \prod_{s\in S} U_s$ and note that this is an open neighbourhood of $\gamma$ in $S\cV_2$ (indeed, it is a basic open set of the form mentioned in Definition \ref{defn:groupoid-svg}). Noting that $r|_U = \prod_{s\in S} r_s|_{U_s}$ and similarly for $s|_U$, it follows that $r|_U$ and $s|_U$ are both homeomorphisms onto their images. Finally, we must check that $r(U)$ and $s(U)$ are open in $(S\cV_2)^{(0)} = \prod_S \cV_2^{(0)}$. We have $r(U) = \prod_{s\in S} r_s(U_s)$ with all but finitely many of the $r_s(U_s)$ equal to $\cV_2^{(0)}$ and each $r_s(U_s)$ open in $\cV_2^{(0)}$, so $r(U)$ is open in the product topology on $\prod_S \cV_2^{(0)}$. Exactly the same argument shows that $s(U)$ is also open in the product topology.
\end{proof}

\begin{rem}
If we took $S\cV_2$ to be simply the product groupoid $\prod_S \cV_2$, the proof of the lemma above would not work, since an infinite product of local homeomorphisms need not be a local homeomorphism in general.

Alternatively, one could wonder whether we could instead take the product $\prod_S \cV_2$ in the box topology instead of the product topology. In this case it would be true that the source and range maps are local homeomorphisms, because arbitrary box products of local homeomorphisms are local homeomorphisms. However, the unit space of $\prod_S \cV_2$ in the box topology would be $\prod_S \cC$ in the box topology (where $\cC$ is the Cantor space). In the case where $S$ is (countably) infinite, this space is not locally compact, by \cite[Theorem 1.3]{Williams1984}, so it cannot be the unit space of a topological groupoid, by our convention in Definition \ref{def:top-gropuoid}.
\end{rem}

Since $S\cV_2$ is ample, it admits a basis for its topology consisting of compact open bisections (see the paragraph just after Definition \ref{def:ample}). In fact, we may describe such a basis as follows. Let $F \subseteq S$ be a finite subset, let $U_s \subset \cV_2$ be any compact open bisection for $s\in F$ and set $U_s = \cV_2^{(0)}$ for $s \in S \smallsetminus F$. Then $\prod_{s \in S} U_s$ is a compact open bisection in $S\mathcal{V}_2$. Since the topology of $\cV_2$ admits a basis consisting of compact open bisections (because it is ample), the compact open bisections of $S\cV_2$ constructed in this way provides a basis for its topology.

\begin{lem}
\label{lem:SV2-is-effective}
The ample groupoid $S\cV_2$ is effective.
\end{lem}
\begin{proof}
We adapt to the setting of $S\cV_2$ the argument from the proof of Lemma \ref{lem:groupdV2-full-V} of the fact that $\cV_2$ is effective. (The fact that $S\cV_2$ is effective does not follow formally from the fact that $\cV_2$ is effective, as explained in Remark \ref{rmk:SV2-effective} below.)

Recall that elements of $S\cV_2$ are collections of triples $(y_s,n_s,x_s)_{s \in S}$ where we have $x_s,y_s \in \{0,1\}^\bN = \cC$, $n_s \in \bZ$, the infinite words $x_s$ and $y_s$ are eventually equal after a relative shift by $n_s$, and for all but finitely many indexing elements $s \in S$ we have $(y_s,n_s,x_s) = (x_s,0,x_s)$. This is topologised as a subspace of the product space $(\cC \times \bZ \times \cC)^S = \cC^S \times \bZ^S \times \cC^S$. In fact, since all but finitely many of the $n_s$ are zero, it is a subspace of the product space $\cC^S \times \bZ^* \times \cC^S$, where $\bZ^*$ denotes the space of functions $\mathrm{Fun}_f(S,\bZ)$ of finite support, i.e., that send all but finitely many elements of $S$ to zero. Although $\bZ^S$ is not discrete unless $S$ is finite (when $S$ is countably infinite it is homeomorphic to the Baire space), its subspace $\bZ^*$ is discrete, and so $S\cV_2$ naturally splits as the disjoint union of the clopen subspaces
\[
S\cV_2[m] = \{ (y_s,n_s,x_s)_{s \in S} \in S\cV_2 \mid n_s = m_s \text{ for all } s \in S \}
\]
indexed by $m = (m_s)_{s \in S} \in \bZ^*$. The isotropy subgroupoid of $S\cV_2$ is
\[
\mathrm{Iso}(S\cV_2) = \{ (y_s,n_s,x_s)_{s \in S} \in S\cV_2 \mid y_s = x_s \text{ for all } s \in S \} ,
\]
and the splitting of $S\cV_2$ induces a splitting of $\mathrm{Iso}(S\cV_2)$ as the disjoint union of the clopen subspaces
\[
\mathrm{Iso}(S\cV_2)[m] = S\cV_2[m] \cap \mathrm{Iso}(S\cV_2),
\]
again indexed by $m = (m_s)_{s \in S} \in \bZ^*$.

We must prove that the interior of $\mathrm{Iso}(S\cV_2)$ in $S\cV_2$ is equal to the unit space $(S\cV_2)^{(0)}$, which is:
\[
(S\cV_2)^{(0)} = \{ (y_s,n_s,x_s)_{s \in S} \in S\cV_2 \mid y_s = x_s \text{ and } n_s = 0 \text{ for all } s \in S \} .
\]
The splittings allow us to compute the interior component-wise. It is clear from the descriptions above that $(S\cV_2)^{(0)} = \mathrm{Iso}(S\cV_2)[0]$, where $0 \in \bZ^*$ denotes the identically-zero function $S \to \bZ$. Since the unit space $(S\cV_2)^{(0)}$ is always an open subspace in any étale groupoid, it follows that the interior of $\mathrm{Iso}(S\cV_2)[0]$ in $S\cV_2[0]$ is precisely $(S\cV_2)^{(0)}$. It therefore remains to prove for any non-zero $m \in \bZ^*$ that $\mathrm{Iso}(S\cV_2)[m]$, as a subspace of $S\cV_2[m]$, has empty interior.

Let us therefore fix any non-zero $m \in \bZ^*$ and choose an element $s_0 \in S$ so that $m_{s_0} \neq 0 \in \bZ$. Suppose that $U$ is an open subset of $S\cV_2[m]$ that is contained in $\mathrm{Iso}(S\cV_2)[m]$; our task is to prove that $U$ is empty. For this purpose, we may assume that $U$ is a basic open subset, so we may assume that it is of the form $U = \prod_{s \in S} U_s$ for open subspaces $U_s \subseteq \cV_2$. (We also have that $U_s = \cV_2^{(0)}$ for all but finitely many $s \in S$, but we will not need this.) We will prove that $U_{s_0}$ is empty, which will imply that $U$ is empty.

To show that $U_{s_0}$ is empty, we first recall from the proof of Lemma \ref{lem:groupdV2-full-V} the following property of the topology of $\cV_2$: every open subset of $\cV_2$ is either empty or uncountable. This follows from the fact that every open subset of the Cantor space $\cC$ is either empty or uncountable, since $\cV_2$ is locally homeomorphic to $\cC$. Thus it will suffice to prove that $U_{s_0}$ is countable. To see this, note that the elements of $U_{s_0}$ are all of the form $(x,m_{s_0},x)$ for $x \in \cC = \{0,1\}^\bN$, since $U$ is contained in $\mathrm{Iso}(S\cV_2)[m]$. Since $m_{s_0} \neq 0$, the infinite word $x$ must be eventually $m_{s_0}$-periodic. Thus $(x,m_{s_0},x) \mapsto x$ is an injection of $U_{s_0}$ into the set of eventually $m_{s_0}$-periodic infinite words in $\{0,1\}$. But there are only countably many such words.
\end{proof}

\begin{rem}
\label{rmk:SV2-effective}
One cannot simply formally deduce effectiveness of $S\cV_2$ from effectiveness of $\cV_2$ in the case when $S$ is infinite. Directly from the definitions, we see that $\mathrm{Iso}(S\cV_2) = S\cV_2 \cap \left( \prod_S \mathrm{Iso}(\cV_2) \right)$, where the intersection is taken inside the product $\prod_S \cV_2$. Since $\cV_2$ is effective, we know that the interior of $\mathrm{Iso}(\cV_2)$ in $\cV_2$ is equal to $\cV_2^{(0)}$, which implies by properties of the product topology that the interior of $\prod_S \mathrm{Iso}(\cV_2)$ in $\prod_S \cV_2$ is equal to $\prod_S \cV_2^{(0)}$. Let us temporarily write $X = \prod_S \cV_2$, $A = S\cV_2$, $B = \prod_S \mathrm{Iso}(\cV_2)$ and $C = \prod_S \cV_2^{(0)}$. In this notation, we have $\mathrm{Int}_X(B) = C$, and to prove effectiveness of $S\cV_2$ we would need to deduce that $\mathrm{Int}_A(A \cap B) = C$. However, in general we have $\mathrm{Int}_A(A\cap B) = A \cap \mathrm{Int}_X(B \cup (X \smallsetminus A))$, which may be larger than $A \cap (\mathrm{Int}_X(B) \cup \mathrm{Int}_X(X \smallsetminus A)) = A \cap \mathrm{Int}_X(B) = A \cap C = C$.
\end{rem}

\subsubsection*{Topological groupoids associated to twisted Brin--Thompson groups.}

Let us return to the setting where we have a group $G$ with a faithful action on a countable set $S$ given by a homomorphism $\alpha \colon G \to \mathrm{Sym}(S)$. The action $\alpha$ induces a left action of $G$ on the topological groupoid $S\cV_2$ by permuting the coordinates. Precisely, viewing elements of $S\cV_2$ as functions $S \to \cV_2$, this action is given by pre-composing with the action $\alpha$ of $G$ on $S$:
\begin{equation}
\label{eq:induced-action}
g \cdot (\gamma \colon S \to \cV_2) = (\gamma \circ \alpha(g^{-1}) \colon S \to \cV_2) .
\end{equation}
(The ${}^{-1}$ is needed on the right-hand side in order for this to be a left action, since pre-composition is a right action.) The formula \eqref{eq:induced-action} for the action may be rewritten in coordinates as:
\begin{equation}
\label{eq:induced-action-coordinates}
g \cdot (\gamma_s)_{s \in S} = (\gamma_{g^{-1}s})_{s \in S} .
\end{equation}
Here, on the right-hand side, we write the left action of an element $h \in G$ on $s \in S$ simply by concatenation as $hs$ rather than $\alpha(h)(s)$. Thus, by Example \ref{ex:semi-direct-product}, we may form the semi-direct product topological groupoid $S\mathcal{V}_2 \rtimes G$.

\begin{rem}
A priori, there is no reason to take the topological groupoid to be $\cV_2$ in the construction. In fact, given a faithful action of $G$ on a countable set $S$, and any topological groupoid $\cG$ with compact unit space, the definition of the twisted topological groupoid $S\cG\rtimes G$ still makes sense. Further properties of these groupoids and their full groups will be studied in \cite{WuZhao25+}.
\end{rem}

\begin{lem}
\label{lem:svg-ample-eff}
The topological groupoid $S\mathcal{V}_2 \rtimes {G}$ is ample and effective and its unit space does not have isolated points.
\end{lem}
\begin{proof}
By Lemma \ref{lem:SV2-etale}, we already know that $S\cV_2$ is ample. As explained in Example \ref{ex:semi-direct-product}, it follows that $S\cV_2 \rtimes G$ is also an ample groupoid. In addition, the unit space of the semi-direct product $S\cV_2 \rtimes G$ is $(S\cV_2)^{(0)} \times \{e\}$, where $e$ denotes the identity element of $G$, so it is also homeomorphic to the Cantor space $\cC$; in particular it has no isolated points.

It remains to prove that $S\cV_2 \rtimes G$ is effective. In other words, we need to verify that the interior of the isotropy subgroupoid
\[
\mathrm{Iso}(S\cV_2 \rtimes G) = \{ (\gamma,g) \in S\mathcal{V}_2\rtimes {G} \mid r((\gamma,g)) = s((\gamma,g)) \}
\]
coincides with the unit space $(S\mathcal{V}_2\rtimes {G})^{(0)}$. Clearly $(S\mathcal{V}_2\rtimes {G})^{(0)}$ is contained in $\mathrm{Iso}(S\cV_2 \rtimes G)$, and it is open in $S\cV_2 \rtimes G$, so it lies in the interior of $\mathrm{Iso}(S\cV_2 \rtimes G)$. We therefore have to show that $\mathrm{Iso}(S\cV_2 \rtimes G)$ has no other interior points.

Recall from Example \ref{ex:semi-direct-product} that $r((\gamma,g)) = (r(\gamma),e)$ and $s((\gamma,g)) = (g^{-1} \cdot s(\gamma),e)$ where $\cdot$ denotes the left action of $G$ on $S\cV_2$ described in \eqref{eq:induced-action-coordinates}. Hence the isotropy subgroupoid may be described as:
\begin{align*}
\mathrm{Iso}(S\cV_2 \rtimes G) &= \{ (\gamma,g) \in S\cV_2 \rtimes G \mid r(\gamma) = g^{-1} \cdot s(\gamma) \} \\
&= \{ ((\gamma_s)_{s \in S},g) \in S\cV_2 \rtimes G \mid (r(\gamma_s))_{s \in S} = g^{-1} \cdot (s(\gamma_s))_{s \in S} \} \\
&= \{ ((\gamma_s)_{s \in S},g) \in S\cV_2 \rtimes G \mid (r(\gamma_s))_{s \in S} = (s(\gamma_{gs}))_{s \in S} \} .
\end{align*}
Since $G$ has the discrete topology, this splits as the disjoint union of the clopen subspaces
\begin{align*}
\mathrm{Iso}(S\cV_2 \rtimes G) \langle g \rangle &= \mathrm{Iso}(S\cV_2 \rtimes G) \cap (S\cV_2 \times \{ g \}) \\
&\cong \{ (\gamma_s)_{s \in S} \in S\cV_2 \mid r(\gamma_s) = s(\gamma_{gs}) \text{ for each } s \in S \} \subset S\cV_2 .
\end{align*}
Let us denote the subspace of $S\cV_2$ described just above by $S\cV_2\langle g \rangle$. In this notation, what we have to show is that the interior of $S\cV_2 \langle e \rangle$ in $S\cV_2$ is equal to $(S\cV_2)^{(0)}$ and that for every $g \in G$ with $g \neq e$ the interior of $S\cV_2 \langle g \rangle$ in $S\cV_2$ is empty.

In the case $g=e$, note that $S\cV_2 \langle e \rangle = \mathrm{Iso}(S\cV_2)$. By Lemma \ref{lem:SV2-is-effective}, the groupoid $S\cV_2$ is effective, so the interior of $S\cV_2 \langle e \rangle = \mathrm{Iso}(S\cV_2)$ as a subspace of $S\cV_2$ is equal to $(S\cV_2)^{(0)}$.

Now let $g \neq e$ and choose an element $s_0 \in S$ such that $gs_0 \neq s_0$ (which exists since the action of $G$ on $S$ is faithful). Suppose that $U$ is an open subset of $S\cV_2$ that is contained in $S\cV_2 \langle g \rangle$; we must prove that $U$ is empty. We may assume that $U$ is a basic open subset, so in particular it is of the form $U = \prod_{s \in S} U_s$ for open subspaces $U_s \subseteq \cV_2$. It will suffice to prove that $U_{s_0} \times U_{gs_0}$ is empty. First note that every element of $U_{s_0} \times U_{gs_0}$ is of the form $((y,m,x),(z,n,y))$ with $x,y,z \in \cC$ and $m,n \in \bZ$, since $U$ is contained in $S\cV_2 \langle g \rangle$. The map
\[
r \times s \colon U_{s_0} \times U_{gs_0} \longrightarrow \cC \times \cC
\]
is open, since it is a product of local homeomorphisms. In particular its image $(r \times s)(U_{s_0} \times U_{gs_0})$ is open in $\cC \times \cC$. But $(r \times s)(U_{s_0} \times U_{gs_0})$ is also contained in the diagonal $\Delta(\cC)$. The diagonal $\Delta(\cC)$ has empty interior in the product $\cC \times \cC$ (since $\cC$ has no isolated points), so $(r \times s)(U_{s_0} \times U_{gs_0})$ must be empty, and hence $U_{s_0} \times U_{gs_0}$ must be empty as well.
\end{proof}

Recall from Example \ref{ex:semi-direct-product} that, given an element $(\gamma,g) \in S\mathcal{V}_2\rtimes {G}$, the range and source maps are given by $r((\gamma,g)) = (r(\gamma),e)$ and $s((\gamma,g)) = (g^{-1} \cdot s(\gamma),e)$, where by an abuse of notation we use $r$ and $s$ to denote the range and source maps of both $S\cV_2$ and $S\cV_2 \rtimes G$.

For the next two results, we will need to understand in more detail what the compact open bisections of $S\cV_2 \rtimes G$ look like. By definition, using the above description of the range and source maps of $S\cV_2 \rtimes G$ and the fact that $G$ has the discrete topology, we see that:
\begin{itemize}
    \item if $\cU$ is a compact open bisection of $S\cV_2$ and $g \in G$ then $\cU \times \{g\}$ is a compact open bisection of $S\cV_2 \rtimes G$;
    \item any compact open bisection of $S\cV_2 \rtimes G$ is a finite disjoint union of subsets $\cU \times \{g\}$ as in the previous point.
\end{itemize}
Now every compact open bisection $\cU$ of $S\mathcal{V}_2$ is of the form $\cU = \prod_{s \in S} \cU_s$, where each $\cU_s$ is a compact open bisection of $\cV_2$ and $\cU_s = \cV_2^{(0)}$ for all but finitely many $s$. Finally, we recall that all compact open bisections of $\cV_2$ may be written as finite disjoint unions of \emph{standard compact open bisections}, which have the following form:
\begin{align*}
& \left\lbrace (y, \lvert w \rvert - \lvert w' \rvert ,x) \in \cV_2 \Bigm| x\in w\{0,1\}^{\bN}, y\in w'\{0,1\}^{\bN},  \rho^{\lvert w' \rvert}(y) = \rho^{\lvert w \rvert}(x) \right\rbrace \\
&\hspace{2cm} = \left\lbrace (w'z , \lvert w \rvert - \lvert w' \rvert , wz) \in \cV_2 \Bigm| z \in \{0,1\}^\bN \right\rbrace ,
\end{align*}
where $w$ and $w'$ are words in $\{0,1\}^{\ast}$.

\begin{prop}
\label{prop:groupoid-full-grp-svg}
The topological full group of $S\mathcal{V}_2 \rtimes G$ is the twisted Brin--Thompson group $S\mathrm{V}_G$.
\end{prop}

\begin{rem}
The twisted Brin--Thompson group $S\mathrm{V}_G$ is known to be a \emph{full group} (in the sense of \cite[Definition 2.24]{BBMZ23}) in the homeomorphism group of the Cantor space: see \cite[top of p.\ 56]{BBMZ23}. However, to calculate its homology using our approach, one needs more than this: it must be realised as the topological full group of an explicit topological groupoid, as in Proposition \ref{prop:groupoid-full-grp-svg}.
\end{rem}

\begin{proof}[Proof of Proposition \ref{prop:groupoid-full-grp-svg}]
By Lemma \ref{lem:svg-ample-eff}, the topological groupoid $S\mathcal{V}_2 \rtimes {G}$ is ample and effective. It is also clear that it is Hausdorff (it is a subspace of the product $(\cC \times \bZ \times \cC)^S \times G$), so by Lemma \ref{lem:full-group-faithful} it is naturally a subgroup of $\Homeo(\cC^S)$. The twisted Brin--Thompson group $S\mathrm{V}_G$ is by definition a subgroup of $\Homeo(\cC^S)$. We will show that they are equal as subgroups.

\textbf{1.} $F(S\mathcal{V}_2 \rtimes {G}) \supseteq S\mathrm{V}_G$.
Let $f\in S\mathrm{V}_G$. By definition, this means that there are two finite collections of dyadic bricks $\{B(\varphi_{1}), \ldots, B(\varphi_{n})\}$ and $\{B(\psi_{1}), \ldots, B(\psi_{n})\}$ forming dyadic partitions of the Cantor space $\cC^S$, together with elements $g_1, \ldots, g_n\in G$, such that $f$ maps each $B(\psi_i) \to B(\varphi_i)$ via the following composition of maps:
\[
f|_{B(\psi_i)} \colon B(\psi_i) \xrightarrow[\cong]{\; h_{\psi_i}^{-1} \;} \cC^S \xrightarrow[]{\; \tau_{g_i} \;} \cC^S \xrightarrow[\cong]{\; h_{\varphi_i} \;} B(\varphi_i),
\]
where $\tau_{g_i}$ is the twist homeomorphism \eqref{eq:tau-gamma} permuting the coordinates of $\cC^S$ via the action of $g_i \in G$ on the coordinate set $S$.

We need to prove that such a local homeomorphism represents a compact open bisection in $S\cV_2\rtimes G$. Recall that, by definition,
\begin{align*}
B(\psi_i) &= \left\lbrace \kappa \in \cC^S \mid \psi_i(s) \text{ is a prefix of of } \kappa(s) \text{ for each }s\in S \right\rbrace \\
&= \prod_{s \in S} B_s(\psi_i) = \prod_{s \in S} \left\lbrace x \in \cC \mid \psi_i(s) \text{ is a prefix of of } x \right\rbrace = \prod_{s \in S} \psi_i(s)\{0,1\}^\bN,
\end{align*}
and $\psi_i \colon S\to \{0,1\}^\ast$ has finite support, i.e., $\psi_i(s)$ is the empty word for all but finitely many elements of $S$. The homeomorphism $h_{\psi_i}^{-1}$ can therefore be represented by the product of shift homeomorphisms
\[
\prod_{s\in S} \rho^{\lvert \psi_i(s) \vert} .
\]
Note that $\rho^{\lvert \psi_i(s) \vert} = \mathrm{id}$ for all but finitely many $s\in S$. The same reasoning also applies to the homeomorphism $h^{-1}_{\varphi_i}$. Hence $f|_{B(\psi_i)}$ can be written as the following composition of maps (from left to right):
\[
\prod_{s \in S} B_s(\psi_i) \xrightarrow[\cong]{\; \prod_{s\in S} \rho^{\lvert \psi_i(s) \rvert} \;} \cC^S \xrightarrow[]{\; \tau_{g_i} \;}\cC^S \xleftarrow[\cong]{\prod_{s\in S} \rho^{\lvert \varphi_i(s) \rvert}} \prod_{s \in S} B_s(\varphi_i).
\]
Redistricted to each coordinate $s$, both $\rho^{\lvert \psi_i(s) \rvert}$ and $ \rho^{\lvert \phi_i(s) \rvert}$ are homeomorphisms, and the standard clopen subspaces $B_s(\psi_i) = \psi_i(s)\{0,1\}^\bN$ and $B_s(\varphi_i) = \varphi_i(s)\{0,1\}^\bN$ are compact open subspaces of $\cC$. Also, $\tau_{g_i}$ is a fixed homeomorphism of $\cC^S$ given by a permutation of coordinates. Restricted to one coordinate $s$, we therefore have the composition of maps
\[
B_s(\psi_i) \xrightarrow[\cong]{\; \rho^{\lvert \psi_i(s) \rvert} \;} \cC_s \xrightarrow[]{\; \mathrm{id} \;} \cC_{g_i s} \xrightarrow[\cong]{\; (\rho^{\lvert \varphi_i(g_i s) \rvert})^{-1} \;} B_{g_i s}(\varphi_i),
\]
where $\cC_s = \cC$ denotes the copy of the Cantor set in the $s$ coordinate and $\mathrm{id} \colon \cC_s \to \cC_{g_i s}$ is the identity map under these identifications. This homeomorphism is realised by the compact open bisection
\[
\cW_i = \prod_{s \in S} \cU_s \times \{g_i\}
\]
of $S\cV_2 \rtimes G$, where for each $s \in S$, the (standard) compact open bisection $\cU_s$ of $\cV_2$ is the following:
\begin{align*}
\cU_s &= \left\lbrace (y, \lvert \varphi_i(s) \rvert - \lvert \psi_i(g_i^{-1}s) \rvert ,x) \in \cV_2 \Bigm| y\in B_{g_i^{-1}s}(\psi_i), x\in B_{s}(\varphi_i), \rho^{\lvert \varphi_i(s) \rvert}(x) = \rho^{\lvert \psi_i(g_i^{-1}s) \rvert}(y) \right\rbrace \\
&= \left\lbrace \Bigl( \psi_i(g_i^{-1}s)z \, , \, \lvert \varphi_i(s) \rvert - \lvert \psi_i(g_i^{-1}s) \rvert \, , \, \varphi_i(s)z \Bigr) \in \cV_2 \Bigm| z \in \{0,1\}^\bN \right\rbrace .
\end{align*}
Note that, since $\varphi_i$ and $\psi_i$ have finite support, this is equal to $\cV_2^{(0)}$ for all but finitely many coordinates $s \in S$, so this is indeed a compact open bisection of $S\cV_2 \rtimes G$ (see the discussion before the proposition). Finally, taking the disjoint union over $i \in \{1,\ldots,n\}$, we have a compact open bisection
\[
\coprod_{i=1}^n \cW_i
\]
of $S\cV_2 \rtimes G$ realising the homeomorphism $f$, so $f \in F(S\mathcal{V}_2 \rtimes G)$.

\textbf{2.} $F(S\mathcal{V}_2 \rtimes {G}) \subseteq S\mathrm{V}_G$.
Now let $f \in F(S\mathcal{V}_2 \rtimes {G})$. By the discussion of compact open bisections of $S\cV_2 \rtimes G$ just before the proposition, this means that $f$ can be written as (the homeomorphism of $\cC^S$ corresponding to) a finite disjoint union of compact open bisections of the form $\cW = \prod_{s \in S} \cU_s \times \{g\}$, where $\cU_{s} \subseteq \cV_2$ is a compact open bisection of $\cV_2$ of the form:
\begin{align*}
\cU_s &= \left\lbrace (y, \lvert w_s \rvert - \lvert w'_s \rvert ,x) \in \cV_2 \Bigm| x\in w_s \{0,1\}^\bN, y\in w'_s \{0,1\}^\bN, \rho^{\lvert w'_s \rvert} (y) = \rho^{\lvert w_s \rvert}(x) \right\rbrace \\
&= \left\lbrace (w'_s z \, , \, \lvert w_s \rvert - \lvert w'_s \rvert \, , \, w_s z) \in \cV_2 \Bigm| z \in \{0,1\}^\bN \right\rbrace ,
\end{align*}
where $w_s, w'_s\in \{0,1\}^\ast$ for each $s\in S$ and $w_s = w'_s = \varnothing$ for all but finitely many $s\in S$. In particular, $w_s$ and $w'_s$ may be regarded as functions $w,w' \colon S \to \{0,1\}^\ast$ with finite support, so we may consider the associated dyadic bricks $B(w)$ and $B(w')$. We also define $\bar{w} \colon S \to \{0,1\}^{\ast}$ by $\bar{w}_s = w_{gs}$ and consider its associated dyadic brick $B(\bar{w})$. The homeomorphism associated to the compact open bisection $\cW$ is therefore the following composition (from right to left):
\[
B(w') = \prod_{s \in S} w'_s\{0,1\}^\bN \xleftarrow[\cong]{\; r \;} \cW \xrightarrow[\cong]{\; s \;} \prod_{s \in S} w_s\{0,1\}^\bN = B(w) \xrightarrow[\cong]{\; g^{-1} \;} \prod_{s \in S} w_{gs}\{0,1\}^\bN = B(\bar{w}).
\]
The additional action of $g^{-1}$ on the right-hand side arises because of the formula $s((\gamma,g)) = (g^{-1} \cdot s(\gamma),e)$ for the source map of $S\cV_2 \rtimes G$. Reversing each of these maps and identifying $\cW$ with $\cC^S$ (using the evident identifications of each $\cU_s$ with $\cC$), we may rewrite this as follows:
\[
B(w') \xrightarrow[\cong]{\; \prod_{s \in S} \rho^{\lvert w'_s \rvert} \;} \cC^S \xleftarrow[\cong]{\; \prod_{s \in S} \rho^{\lvert w_s \rvert} \;} B(w) \xleftarrow[\cong]{\; g \;} B(\bar{w}).
\]
We may rewrite the homeomorphism $\cC^S \leftarrow B(\bar{w})$ above as follows:
\begin{align*}
\prod_{s\in S} \rho^{\lvert w_s \rvert} \circ g &= g \circ \left( g^{-1} \circ \prod_{s\in S} \rho^{\lvert w_s \rvert} \circ g \right) \\
&= g \circ \prod_{s \in S} \rho^{\lvert \bar{w}_s \rvert} ,
\end{align*}
and hence the homeomorphism associated to the compact open bisection $\cW$ may be rewritten as follows:
\[
B(w') \xrightarrow[\cong]{\; \prod_{s \in S} \rho^{\lvert w'_s \rvert} \;} \cC^S \xrightarrow[\cong]{\; \tau_{g^{-1}} \;} \cC^S \xleftarrow[\cong]{\; \prod_{s \in S} \rho^{\lvert \bar{w}_s \rvert} \;} B(\bar{w}),
\]
where we now write $\tau_{g^{-1}}$ instead of $g^{-1}$ for the middle homeomorphism, following the notation \eqref{eq:tau-gamma}. We now note that this is precisely the canonical twist homeomorphism associated to the dyadic bricks $B(\bar{w})$, $B(w')$ and the group element $g^{-1}$, as described at the beginning of \S\ref{subsection:twisted-BT}.

Since the domain and image of $f$ are the whole Cantor cube $\cC^S$, it follows that, as we run through the finite disjoint union of compact open bisections $\cW$ corresponding to $f$ as above, their sources and ranges each form a dyadic partition of $\cC^S$ and $f$ is defined on each dyadic brick in the source partition by a canonical twist homeomorphism as above. This means, by definition, that $f$ lies in $S\mathrm{V}_G$.
\end{proof}

\begin{lem}
\label{lem:svg-purely-inf-min}
The ample groupoid $S\mathcal{V}_2 \rtimes {G}$ is purely infinite minimal.
\end{lem}
\begin{proof}
Recall that the unit space of $S\mathcal{V}_2 \rtimes {G}$ is $\cC^S \times \{e\}$, where $e$ denotes the identity element of $G$. This is homeomorphic to the Cantor space $\cC$, which is compact. Now let $U$ and $W \neq \emptyset$ be two compact open subsets of the unit space. We need to find a compact open bisection $\sigma \subseteq S\cV_2 \rtimes G$ such that $s(\sigma) = U$ and $r(\sigma) \subseteq W$.

Any compact open subspace of $\cC^S$ is a finite disjoint union of dyadic bricks, so we may choose coverings of $U$ and $W$ by finitely many pairwise disjoint dyadic bricks, which we denote by $B(\varphi_1),B(\varphi_2),\ldots B(\varphi_n)$ and $B(\psi_1),B(\psi_2),\ldots B(\psi_m)$ respectively, associated to finitely-supported functions $\varphi_i,\psi_i \colon S \to \{0,1\}^\ast$. Up to subdivision, we may assume that $m\geq n$. It suffices now to construct compact open bisections $\sigma_i$ of $S\cV_2 \rtimes G$ such that $s(\sigma_i) = B(\varphi_i)$ and $r(\sigma_i) = B(\psi_i)$ for each $1\leq i\leq n$. Let
\[
\cU_{i,s} = \left\lbrace (y, \lvert \varphi_i(s) \rvert - \lvert \psi_i(s) \rvert ,x) \in \cV_2 \Bigm| x \in \varphi_i(s)\{0,1\}^{\bN}, y\in \psi_i(s)\{0,1\}^{\bN}, \rho^{\lvert \psi_i(s) \rvert}(y) = \rho^{\lvert \varphi_i(s) \rvert}(x) \right\rbrace
\]
for each $s \in S$ and $1\leq i\leq n$. Then for each $1\leq i\leq n$ the subset
\[
\sigma_i = \prod_{s \in S} \cU_{i,s} \times \{e\} \subset S\cV_2 \rtimes G
\]
is a compact open bisection with the required source and range.

Finally, since the dyadic bricks $B(\varphi_1),B(\varphi_2),\ldots B(\varphi_n)$ are pairwise disjoint, and similarly for the dyadic bricks $B(\psi_1),B(\psi_2),\ldots B(\psi_m)$, we may take the disjoint union $\sigma = \sigma_1 \sqcup \cdots \sqcup \sigma_n$ to obtain another compact open bisection $\sigma \subseteq S\cV_2 \rtimes G$ that has $s(\sigma) = U$ and $r(\sigma) = B(\psi_1) \sqcup \cdots \sqcup B(\psi_n) \subseteq W$.
\end{proof}

Lemmas \ref{lem:svg-ample-eff} and \ref{lem:svg-purely-inf-min} together imply that the topological groupoid $S\mathcal{V}_2 \rtimes {G}$ satisfies the conditions of Theorem \ref{thm:homgy-groupoid-fullgroup}. It remains to show that the homology of $S\mathcal{V}_2 \rtimes {G}$ vanishes in every degree. We first calculate the homology of the groupoid $S\mathcal{V}_2$.

\begin{lem}
\label{lem:groupoid-sv-acyclic}
We have $H_i(S\mathcal{V}_2;\mathbb{Z}) = 0$ for any $i\geq 0$.
\end{lem}
\begin{proof}
Let $S' = S \smallsetminus \{s\}$ for some $s \in S$. (We may assume that $S$ is non-empty since otherwise $S\cV_2 = \emptyset\cV_2$ is the trivial groupoid.) Notice that the groupoid $\cV_2 \times S'\cV_2$ is isomorphic to $S\cV_2$. Since $H_i(\mathcal{V}_2;\mathbb{Z}) = 0$ for any $i\geq 0$, the K\"unneth formula \cite[Theorem 2.4]{Matui2016} implies that $H_i(S\mathcal{V}_2;\mathbb{Z}) = 0$ for any $i\geq 0$.
\end{proof}

\begin{prop}
\label{prop:groupoi-svg-cyclic}
We have $H_i(S\mathcal{V}_2\rtimes {G};\mathbb{Z}) = 0$ for any $i\geq 0$.
\end{prop}
\begin{proof}
With the help of Lemma \ref{lem:groupoid-sv-acyclic}, this calculation follows immediately from the spectral sequence in \cite[Theorem 3.8(2)]{Matui2012}. Namely, this spectral sequence is of the form:
\[
E^2_{p,q} = H_p(G; H_q(S\cV_2;\bZ)) \Rightarrow H_{p+q}(S\cV_2\rtimes G;\bZ).
\]
From the previous lemma we know that $H_q(S\mathcal{V}_2; \mathbb{Z}) = 0$ for any $q\geq 0$, so on the $E^2$ page of the spectral sequence we have $E_{p,q}^2 = H_p(G; H_q(S\cV_2;\bZ)) = H_p(G;0) = 0$ for any $p,q\geq 0$. Thus the entire spectral sequence vanishes and we have $H_i(S\mathcal{V}_2\rtimes {G}; \mathbb{Z}) = 0$ for any $i\geq 0$.
\end{proof}

\begin{thm}[Theorem \ref{thm:acyclic-t-br-thom}]
\label{thm:tbt-acyc}
For any group $G$ acting faithfully on a countable set $S$, the associated twisted Brin--Thompson group $S\mathrm{V}_G$ is acyclic.
\end{thm}
\begin{proof}
As noted above, Lemmas \ref{lem:svg-ample-eff} and \ref{lem:svg-purely-inf-min} imply that the topological groupoid $S\cV_2 \rtimes G$ satisfies the conditions of Theorem \ref{thm:homgy-groupoid-fullgroup}. By Proposition \ref{prop:groupoi-svg-cyclic}, we have $H_i(S\mathcal{V}_2 \rtimes G; \mathbb{Z}) = 0$ for any $i\geq 0$, so Theorem \ref{thm:homgy-groupoid-fullgroup} tells us that its topological full group $F(S\mathcal{V}_2 \rtimes G)$ is acyclic. But by Proposition \ref{prop:groupoid-full-grp-svg}, this is the twisted Brin--Thompson group $S\mathrm{V}_G$. Thus $S\V_G$ is acyclic.
\end{proof}

\begin{rem}
\label{rmk:uncountable-S}
The only way in which we used the assumption that $S$ is countable in this section was in order to have $\cC^S \cong \cC$, which was used to deduce that the unit spaces of $S\cV_2$ and $S\cV_2 \rtimes G$ (which are both homeomorphic to $\cC^S$) have the appropriate point-set topological properties, namely that they are:
\begin{itemize}
    \item Hausdorff,
    \item totally disconnected,
    \item without isolated points,
    \item locally compact.
\end{itemize}
However, if $S$ is an uncountable set, then the space $\cC^S$ still satisfies all of these hypotheses, even though $\cC^S \not\cong \cC$ in this case.\footnote{For uncountable $S$ we have $\cC^S \not\cong \cC$ because $\cC^S$ is not second countable (indeed $X^S$ is not second countable for any non-indiscrete space $X$). If $S$ has at least continuum cardinality (a stronger condition unless assuming the Continuum Hypothesis), then we moreover have $\lvert \cC^S \rvert > \lvert \cC \rvert$.} For the first three properties above, this is because they are preserved under taking arbitrary products. Although the last property (local compactness) is not preserved under taking arbitrary products, the combined property of being locally compact \emph{and} compact is preserved under taking arbitrary products. Thus, since $\cC$ is compact and satisfies the above properties, all of its powers $\cC^S$ also satisfy these properties. The proofs in this section are therefore valid also for any uncountable set $S$; in particular, Theorem \ref{thm:acyclic-t-br-thom} is also true in this case.

The proof of simplicity of $S\V_G$ due to Belk--Zaremsky \cite[Theorem 3.4]{BelkZaremsky22} depends on countability of $S$ in order to have $\cC^S \cong \cC$, since it uses a criterion of Bleak--Elliott--Hyde \cite[Theorem 4.18]{BleakElliottHyde2024} for simplicity of subgroups of $\mathrm{Homeo}(\cC)$. However, it is still true that $S\V_G$ is simple when $S$ is uncountable.
\end{rem}

\begin{lem}
\label{lem:SVG-simple}
For any group $G$ acting faithfully on a set $S$, the group $S\V_G$ is simple.
\end{lem}

When $S$ is countable, this is precisely \cite[Theorem 3.4]{BelkZaremsky22}. The argument below, deducing the general case from the countable case, is due to Matthew Zaremsky (personal communication).

\begin{proof}[Proof of Lemma \ref{lem:SVG-simple}]
Each element $\alpha \in S\V_G$ is a composition of finitely many canonical twist homeomorphisms, each of which involves only finitely many elements of $G$ and finitely many dyadic bricks, each of which is supported on finitely many coordinates of $S$. Denote by $H \subseteq G$ the subgroup generated by the relevant elements of $G$, and denote by $T \subseteq S$ the union of the $H$-orbits of the relevant coordinates of $S$. Then $H$ and $T$ are both countable and $\alpha \in T\V_H \subseteq S\V_G$, where we are viewing $T\V_H$ as a subgroup of $S\V_G$ by sending an element $\alpha$ of $T\V_H \subseteq \mathrm{Homeo}(\cC^T)$ to the element $\alpha \times \mathrm{id}_{S \smallsetminus T}$ of $S\V_G \subseteq \mathrm{Homeo}(\cC^S)$, where $\mathrm{id}_{S \smallsetminus T}$ denotes the identity map of $\cC^{S \smallsetminus T}$. By the same argument, any given countable collection of elements of $S\V_G$ is contained in a subgroup of the form $T\V_H \subseteq S\V_G$ with $H$ and $T$ countable.

Now let $\alpha$ and $\beta$ be any two non-trivial elements of $S\V_G$. To prove that $S\V_G$ is simple it will suffice to show that $\beta \in \langle \alpha \rangle^{S\V_G}$, i.e.\ that $\beta$ is contained in the normal closure of $\alpha$ in $S\V_G$. By the previous paragraph, we know that $\alpha,\beta \in T\V_H \subseteq S\V_G$ for some countable subgroup $H \subseteq G$ and some countable $H$-invariant subset $T \subseteq S$. By \cite[Theorem 3.4]{BelkZaremsky22} we know that $T\V_H$ is simple, and hence $\beta \in \langle \alpha \rangle^{T\V_H}$. But $\langle \alpha \rangle^{T\V_H} \subseteq \langle \alpha \rangle^{S\V_G}$, and so we also have $\beta \in \langle \alpha \rangle^{S\V_G}$.
\end{proof}

\bibliographystyle{alpha}
\bibliography{references.bib}

\end{document}